\documentclass[11pt]{article}
\usepackage[top=2cm, bottom=2cm, left=2.5cm, right=2.5cm] {geometry}
\usepackage{comment}
\usepackage{amsmath,amssymb,theorem,color}
\numberwithin{equation}{section} 
\usepackage{amssymb}
\usepackage{color}
\usepackage{amsmath}
\usepackage{graphicx}
\usepackage{amsfonts}%
\newcommand{\R}{\ensuremath{\mathbb{R}}}

\newcommand{\N}{\ensuremath{\mathbb{N}}}
\newcommand{\eps}{\varepsilon}

\newcommand{\cC}{\mathcal{C}}
\newcommand{\cF}{\mathcal{F}}
\newcommand{\bP}{\mathbb{P}}

\newcommand{\mP}{\mathbb{P}}

\newcommand{\E}{\mathbb{E}}

\newcommand{\cD}{\mathcal{D}}

\newcommand{\cB}{\mathcal{B}}

\newcommand{\ltn}{\ensuremath{\left| \! \left| \! \left|}}
\newcommand{\rtn}{\ensuremath{\right| \! \right| \! \right|}}

\newtheorem{theorem}{Theorem}[section]
{ \theorembodyfont{\normalfont} 
	\newtheorem{example}[theorem]{Example}
	\newtheorem{remark}[theorem]{Remark}
}

\newtheorem{lemma}[theorem]{Lemma}
\newtheorem{corollary}[theorem]{Corollary}
\newtheorem{proposition}[theorem]{Proposition}

\newcounter{enumctr}

\usepackage[displaymath, mathlines]{lineno}
\begin{document}
\title{Asymptotic stability of controlled differential equations. \\Part I: Young integrals}

\author{Luu Hoang Duc \thanks{Luu H. Duc is with
		Max-Planck-Institute for Mathematics in the Sciences, Leipzig, Germany,
		\& Institute of Mathematics, Viet Nam Academy of Science and Technology	{\tt\small duc.luu@mis.mpg.de, lhduc@math.ac.vn}
	}, $\;$ Phan Thanh Hong \thanks{Phan T. Hong is with
	Thang Long University, Hanoi, Vietnam
	{\tt\small hongpt@thanglong.edu.vn}}
}
\date{}
\maketitle

\begin{abstract}
We provide a unified analytic approach to study stationary states of controlled differential equations driven by rough paths, using the framework of random dynamical systems and random attractors. Part I deals with driving paths of finite $p$-variations with $1 \leq p <2$ so that the integrals are interpreted in the Young sense. Our method helps to generalize recent results \cite{GAKLBSch2010}, \cite{ducGANSch18}, \cite{duchongcong18} on the existence of the global pullback attractors for the generated random dynamical systems. We also prove sufficient conditions for the attractor to be a singleton, thus the pathwise convergence is in both pullback and forward senses.
\end{abstract}



\section{Introduction}
This paper studies the asymptotic behavior of the controlled differential equation
\begin{equation}\label{fSDE0}
dy_t = [Ay_t + f(y_t)]dt + g(y_t)d x_t,\quad t\in \R_+, y(0)=y_0 \in \R^d, 
\end{equation}
where $A \in \R^{d\times d}$, $f: \R^d \to \R^d$ and $g: \R^d \to \R^{d\times m}$ are globally Lipschitz continuous functions, and the driving path $x$ is in the space $\cC^{p{\rm - var}}(\R, \R^m)$ of continuous paths with finite $p$ - variation norm, for some $p \geq 1$. It is well known that such equation can be solved using T. Lyons' theory of rough paths (see \cite{lyons98}, \cite{lyonsetal07} and also \cite{friz}), hence it appears as the path-wise approach to solve the stochastic differential equation 
\begin{equation}\label{stochYDE}
dy_t = [Ay_t + f(y_t)]dt +g(y_t)d Z_t
\end{equation}
where $Z_t$ is a stationary stochastic process $Z_t$ with almost sure all realizations of $\nu$ - H\"older continuous for some $\nu \in (0,1)$ (e.g. fractional Brownian motions \cite{mandelbrot} with Hurst indices $H \in (0,1)$). 

The first part considers the simple case $1 \leq p <2$ so that system \eqref{fSDE0} is understood in the integral form
\begin{equation}\label{YDEintegral}
y_t = y_0 + \int_0^t [Ay_s + f(y_s)]ds + \int_0^t g(y_s)dx_s, \quad \forall t\geq 0, 
\end{equation}
where the second integral is understood in the Young sense \cite{young}. The existence and uniqueness theorem for Young differential equations is proved in many versions, e.g. \cite{lyons94}, \cite{lyons98}, \cite{zahle}, \cite{nualart01}.%

Our aim is to investigate the role of the driving noise in the longterm behavior of system \eqref{stochYDE}. In fact, this question is studied in a probabilistic approach in \cite{hairer03}, \cite{hairer07}, \cite{hairer11}, \cite{hairer13}, in which they prove the ergodicity of system \eqref{stochYDE}, namely under the general dissipativity condition for the drift coefficient function and some additional regularity conditions, there exists a unique adapted stationary solution for \eqref{stochYDE} in the sense that the generated {\it stochastic dynamical system} over a stationary noise process has a unique invariant probability measure \cite{hairer07}. Moreover, the convergence is of probability type, i.e. other probability measures converge to the unique invariant measure in the total variation norm. 

In this paper, we however follow a simpler analytic approach, namely we impose assumptions for the drift coefficient so that there exists a unique equilibrium for the deterministic system $\dot{\mu} = A\mu + f(\mu)$ which is asymptotically stable; and then raise the questions on the asymptotic dynamics of the perturbed system, in particular the existence of stationary states and their asymptotic (stochastic) stability \cite{khasminskii} with respect to almost sure convergence.\\
These questions could be studied in the framework of random dynamical systems \cite{arnold}. Specifically, results in \cite{GAKLBSch2010} and recently in \cite{BRSch17}, \cite{duchongcong18} reveal that the stochastic Young system \eqref{stochYDE} generates a random dynamical system, hence asymptotic structures like random attractors are well-understood. In this scenarios, system \eqref{stochYDE} has no deterministic equilibrium but is expected to possess a random attractor, although little is known on the inside structure of the attractor and much less on whether or not the attractor is a (random) singleton.

We remind the reader of a well-known technique in \cite{Sus78}, \cite{ImkSchm01}, \cite{KelSchm98} to generate RDS and to study random attractors of system \eqref{fSDE0} by a conjugacy transformation $y_t = \psi_g(\eta_t,z_t)$, where the semigroup $\psi_g$ generated by the equation
$\dot{u} = g(u)$ and $\eta$ is the unique stationary solution of the Langevin equation $d\eta = -\eta dt + dZ_t$. The transformed system
\begin{equation}\label{conju}
\dot{z}_t = \Big(\frac{\partial \psi_g}{\partial u}\Big)^{-1}(\eta_t,z_t) \Big[ A\psi_g(\eta_t,z_t) + f(\psi_g(\eta_t,z_t)) + \eta C\psi_g(\eta_t,z_t)  \Big]
\end{equation}
can then be solved in the pathwise sense and the existence of random attractor for \eqref{conju} is equivalent to the existence of random attractor for the original system.
Unfortunately, this conjugacy method only works in some special cases, particularly if $g(\cdot)$ is the identity matrix, or more general if $g(y) = Cy$ for some matrix $C$ that commutes with $A$ (see further details in Remark \ref{linearrem}). For more general cases, the reader is refered to \cite{ducGANSch18}, \cite{duchongcong18} and the references therein for recent methods in studying the asymptotic behavior of Young differential equations . 

Notice that the existence of a random attractor for the generated random dynamical system was studied in \cite{GAKLBSch2010} and \cite{ducGANSch18} for Young differential equations using an additional construction of stopping times based on fractional Brownian motions, but the method only works under a small noise in the sense that its H\"older seminorm is integrable and can be controlled to be small. In contrast, our method uses the semigroup technique and the crucial Lemma \ref{ytest} to estimate the solution norm under a regular discretization scheme. Thanks to Theorem \ref{YDE} and its corollaries, the supremum and $p-$variation norms of the solution of \eqref{fSDE0} on each interval are estimated and depend only on the left end-point of the interval. This makes it possible to apply the discrete Gronwall lemma for the obtained recurrence relation of the solution norms at the end-points, and finally yields a stability criterion in Theorem \ref{attractor}. As such, the method can be applied for a general source of noises, and the stability criterion matches the classical one for ordinary differential equations when the effect of driving noise is cleared. 
Moreover, it could also be extended to rough differential equations using rough path theory, provided that the driving path is of a little higher regularity (H\"older continuity is required) so that the rough integral is understood in the sense of Gubinelli \cite{gubinelli} for controlled rough paths (see details in Part II \cite{ducRDE}). 
 
The paper is organized as follows. Section 2 is devoted to present preliminaries and main results of the paper, where the norm estimates of the solution of \eqref{fSDE0} is then presented in Subsection 2.1. In Subsection 3.1, we introduce the generation of random dynamical system from the equation \eqref{stochYDE}. Using Lemma \ref{ytest}, we prove the existence of a global random pullback attractor in Theorem \ref{attractor}. Finally in Subsection \ref{singletonsub}, we prove that the attractor is both a pullback and forward singleton attractor if $g$ is a linear map in Theorem \ref{linear}, or if $g \in C^2_b$ for small enough Lipschitz constant $C_g$ in Theorem \ref{gbounded}. As an illustration, we analyze in Example \ref{ex1} the asymptotic stability for the inverted pendulum with stochastic excitation.

\section{Preliminaries and main results}
Let us first briefly make a survey on Young integrals. Denote by $\cC([a,b],\R^r)$, for $r\geq 1$, the space of all continuous paths $x:\;[a,b] \to \R^r$ equipped with supremum norm $\|x\|_{\infty,[a,b]}=\sup_{t\in [a,b]} \|x_t\|$, where $\|\cdot\|$ is the Euclidean norm of a vector in $\R^r$. 
For $p\geq 1$ and $[a,b] \subset \R$, denote by $\cC^{p{\rm-var}}([a,b],\R^r)$ the space of all continuous paths $x \in \cC([a,b],\R^r)$ which is of finite $p-$variation, i.e. 
$\ltn x\rtn_{p{\rm-var},[a,b]} :=\left(\sup_{\Pi(a,b)}\sum_{i=1}^n \|x_{t_{i+1}}-x_{t_i}\|^p\right)^{1/p} < \infty$ where the supremum is taken over the whole class $\Pi(a,b)$ of finite partitions $\Pi=\{ a=t_0<t_1<\cdots < t_n=b \}$ of $[a,b]$ (see e.g. \cite{friz}). Then $\cC^{p{\rm-var}}([a,b],\R^r)$, equipped with the $p-$var norm $	\|x\|_{p{\rm -var},[a,b]}:= \|x_a\|+\ltn x\rtn_{p{\rm -var},[a,b]}$, is a nonseparable Banach space \cite[Theorem 5.25, p.\ 92]{friz}. Also for each $0<\alpha<1$, denote by $\cC^{\alpha\rm{-Hol}}([a,b],\R^r)$ the space of H\"older continuous paths with exponent $\alpha$ on $[a,b]$, and equip it with the norm $\|x\|_{\alpha{\rm -Hol},[a,b]}: = \|x_a\| + \sup_{a\leq s<t\leq b}\frac{\|x_t-x_s\|}{(t-s)^\alpha}$.

We recall here a result from \cite[Lemma\ 2.1]{congduchong17}.
\begin{lemma}\label{additive}
Let $x\in \cC^{p{\rm -var}}([a,b],\R^d)$, $p\geq 1$. If $a = a_1<a_2<\cdots < a_k = b$, then 
\[
\sum_{i=1}^{k-1}\ltn x\rtn^p_{p{\rm -var},[a_i,a_{i+1}]}\leq \ltn x\rtn^p_{p-{\rm var},[a_1,a_k]}\leq (k-1)^{p-1}\sum_{i=1}^{k-1}\ltn x\rtn^p_{p{\rm -var},[a_i,a_{i+1}]}.
\] 
\end{lemma}

Now, for $y\in \cC^{q{\rm-var}}([a,b],\R^{d\times m})$ and $x\in \cC^{p{\rm -var}}([a,b],\R^m)$ with  $\frac{1}{p}+\frac{1}{q}  > 1$, the Young integral $\int_a^b y_t dx_t$ can be defined as  $
\int_a^b y_s dx_s:= \lim \limits_{|\Pi| \to 0} \sum_{[u,v] \in \Pi} y_u(x_v-x_u)$, where the limit is taken over all the finite partitions $\Pi$ of $[a,b]$ with $|\Pi| := \displaystyle\max_{[u,v]\in \Pi} |v-u|$ (see \cite[p.\ 264--265]{young}). This integral satisfies the additive property and the so-called {\it Young-Loeve estimate} \cite[Theorem 6.8, p.\ 116]{friz}
\begin{eqnarray}\label{YL0}
\Big\|\int_s^t y_u dx_u-y_s[x_t-x_s]\Big\| \leq (1-2^{1-\frac{1}{p}-\frac{1}{q}})^{-1} \ltn y\rtn_{q{\rm -var},[s,t]} \ltn x\rtn_{p{\rm -var},[s,t]},  \;\forall [s,t]\subset [a,b],
\end{eqnarray}
From now on, we only consider $q = p$ for convenience. We impose the following assumptions on the coefficients $A,f$ and $g$ and the driving path $x$.

{\bf Assumptions}

(${\textbf H}_1$) $A \in \R^{d\times d}$ is a matrix which has all eigenvalues of negative real parts;

(${\textbf H}_2$) $f: \R^d \to \R^d$ and $g: \R^d \to \R^{d\times m},$ are globally Lipschitz continuous functions. In addition, $g \in C^1$ such that $D_g$ is also globally Lipschitz continuous. Denote by $C_f,C_g$ the Lipschitz constants of $f$ and $g$ respectively;

(${\textbf H}_3$) for a given $p \in (1,2)$, $x$ belongs to the space $\cC^{p{\rm - var}}(\R, \R^m)$ of all continuous paths which is of finite $p-$variation on any interval $[a,b]$. In particular, $x$ is a realization of a stationary stochastic process $Z_t(\omega)$ with almost sure all realizations in the space $C^{p{\rm - var}}(\R, \R^m)$, such that 
\begin{equation}\label{Gamma}
\Gamma(p):=\Big(E \ltn Z \rtn^p_{p{\rm -var},[-1,1]}\Big)^{\frac{1}{p}} < \infty. 
\end{equation}
For instance, $Z$ could be an $m-$dimensional fractional Brownian motion $B^H$ \cite{mandelbrot} with Hurst exponent $H \in (\frac{1}{2},1)$, i.e. a family of centered Gaussian processes $B^H = \{B^H_t\}$, $t\in \R$ or $\R_+$ with continuous sample paths and the covariance function
\[
R_H(s,t) = \tfrac{1}{2}(t^{2H} + s^{2H} - |t-s|^{2H}),\quad \forall t,s \in \R.
\] 
Assumption (${\textbf H}_1$) ensures that the semigroup $\Phi(t) =e^{At}, t\in \R$ generated by $A$ satisfies the following properties.
\begin{proposition}\label{A}
	Assume that $A$ has all eigenvalues of negative real parts. Then there exist constants $C_A\geq 1,\lambda_A >0$ such that the generated semigroup $\Phi(t) = e^{At}$ satisfies 
	\begin{eqnarray}
	\|\Phi\|_{\infty,[a,b]} &\leq& C_Ae^{-\lambda_A a}, \label{estphi1}\\
	\ltn\Phi\rtn_{p{\rm -var},[a,b]} &\leq& |A|C_A e^{-\lambda_A a}(b-a),\quad \forall\;  0\leq a<b, \label{estphi2}
	\end{eqnarray}
	in which $|A| : =\displaystyle\sup_{\|x\|=1}\frac{\|Ax\|}{\|x\|}$.
\end{proposition}
\begin{proof}
	The first inequality is due to \cite[Chapter 1, \S3]{Adrianova}. The second one is followed from the mean value theorem
	\begin{eqnarray*}
		\|\Phi(u)-\Phi(v)\| =\left\|\int_u^v A\Phi(s)ds\right\| 
		\leq \int_u^v |A| C_A e^{-\lambda_A s}ds  \leq |A|C_A e^{-\lambda_A a}(v-u),
	\end{eqnarray*}
	for any $u<v$ in $[a,b]$ where $e^{-\lambda_A \cdot}$ is a decreasing function. 	
\end{proof}

Our main results (Theorem \ref{attractor}, Theorem \ref{linear} and Theorem \ref{gbounded}) could be summarized as follows.

\begin{theorem}\label{mainthm}
	Assume that the system \eqref{stochYDE} satisfies the assumptions ${\textbf H}_1-{\textbf H}_3$, and further that  $\lambda_A > C_fC_A$, where $\lambda_A$ and $C_A$ are given from \eqref{estphi1},\eqref{estphi2}. If
	\begin{equation}\label{criterion}
	\lambda_A - C_A C_f> C_A(1+|A|) e^{\lambda_A+2(|A|+C_f)} \Big\{\Big[2(K+1)C_g\Gamma(p)\Big]^p + \Big[2(K+1)C_g\Gamma(p)\Big]\Big\},
	\end{equation}
	where $\Gamma(p)$ is defined in \eqref{Gamma} and $K$ in \eqref{constK}, then the generated random dynamical system $\varphi$ of \eqref{stochYDE} possesses a pullback attractor $\mathcal{A}(x)$. Moreover, in case $g(y) = Cy +g(0)$ is a linear map satisfying \eqref{criterion} or in case $g \in C^2_b$ with the Lipschitz constant $C_g$ small enough, this attractor is a singleton, i.e. $\mathcal{A}(x) = \{a(x)\}$ a.s., thus the pathwise convergence is in both the pullback and forward directions.	
\end{theorem}
For the convenience of the readers, we introduce some notations and constants which are used throughout the paper.
\allowdisplaybreaks
 \begin{eqnarray}
	L &:=& |A|+C_f, \quad L_f := C_AC_f, \quad \lambda := \lambda_A-L_f; \label{L}\\
	K&:=&(1-2^{1-\frac{2}{p}})^{-1}, \alpha := \log (1+ \frac{1}{K+1}); 	\label{constK}\\
	M_0 &:=& \frac{\|f(0)\|}{L}+  \frac{\|g(0)\|}{(K+1)C_g}; \label{M0}\\
	M_1 &:=& KC_A e^{\lambda_A}(1+|A|); \label{M1}\\
	M_2 &:=& \max\left\{C_A\frac{e^{\lambda}-1}{\lambda} , M_1C_g\Big( \frac{1}{L}+  \frac{1}{(K+1)C_g} \Big), M_1C_g \right\} \max \{\|f(0)\|,\|g(0)\|\}; \label{M2}\\
	\hat{G} &:=& C_A e^{\lambda_A}(1+|A|)e^{4L} \Big\{\Big[2(K+1)C_g\Gamma(p)\Big]^p + \Big[2(K+1)C_g\Gamma(p)\Big]\Big\}.\label{Gmu}
\end{eqnarray}
\subsection{Solution estimates}
In this preparatory subsection we are going to estimate several norms of the solution. To do that, the idea is to evaluate the norms of the solution on a number of consecutive small intervals. Here we would like to construct, for any $\gamma >0$ and any given interval $[a,b]$, a sequence of greedy times $\{\tau_k(\gamma)\}_{k \in \N}$ as follows  (see e.g. \cite{cassetal}, \cite{ducGANSch18}, \cite{congduchong17})
\begin{equation}\label{greedytime}
\tau_0 = a, \tau_{k+1}(\gamma) := \inf \{t > \tau_k(\gamma): \ltn x \rtn_{p{\rm -var},[\tau_k(\gamma),t]} = \gamma\} \wedge b.
\end{equation}
Define
\begin{equation}\label{N}
N = N_{\gamma,[a,b]}(x) : = \sup \{k \in \N, \tau_k(\gamma) \leq b\},
\end{equation}
then due to the superadditivity of $\ltn x \rtn^p_{p{\rm -var},[s,t]}$ 
\begin{eqnarray}\label{Nest}
N- 1 &\leq& \sum_{k=0}^{N - 2} \gamma^{-p} \ltn x \rtn^p_{p {\rm -var},[\tau_k,\tau_{k+1}]} \leq \gamma^{-p} \ltn x \rtn^p_{p {\rm -var},[\tau_0,\tau_{N-1}]} \leq \gamma^{-p} \ltn x \rtn^p_{p {\rm -var},[a,b]}, \notag\\
\text{which yields\ }\  N&\leq& 1 + \gamma^{-p} \ltn x \rtn^p_{p {\rm -var},[a,b]}.
\end{eqnarray}
From now on, we fix $p \in (1,2)$ and $\gamma := \frac{1}{2(K+1)C_g}$, and write in short $N_{[a,b]}(x)$ to specify the dependence of $N$ on $x$ and the interval $[a,b]$.

We assume throughout this section that the assumption  (${\textbf H}_2$), (${\textbf H}_3$) are satisfied. The following theorem presents a standard method to estimate the $p-$variation and the supremum norms of the solution of \eqref{fSDE0}, by using the continuous Gronwall lemma and a discretization scheme with the greedy times \eqref{greedytime}. 
\begin{theorem}\label{YDE}
There exists a unique solution to \eqref{fSDE0} for any initial value, whose supremum and $p-$variation norms are estimated as follows
	\begin{eqnarray}
	\|y\|_{\infty,[a,b]} &\leq&  \Big[\|y_a\| + M_0N_{[a,b]}(x)\Big] e^{\alpha N_{[a,b]}(x) +2L (b-a)}, \label{estx} \\
	\| y \|_{p{\rm -var},[a,b]} &\leq& \Big[\|y_a\| +M_0 N_{[a,b]}(x)\Big]e^{\alpha N_{[a,b]}(x)+2L(b-a) }N^{\frac{p-1}{p}}_{[a,b]}(x), \label{estx2}
	\end{eqnarray}
	where $L,\alpha$ and $M_0$ are given by \eqref{L}, \eqref{constK} and \eqref{M0} respectively. 
\end{theorem}

\begin{proof}
There are similar versions of Theorem \ref{YDE} in \cite[Proposition 1]{lejay} for Young equations or in \cite[Lemma 4.5]{cassetal} and \cite[Theorem 3.1]{riedelScheutzow} for rough differential equations with bounded $g$, thus we will only sketch out the proof here for the benefit of the readers. To prove \eqref{estx}, we use the fact that $\ltn g(y)\rtn_{p{\rm-var},[s,t]}\leq C_g\ltn y\rtn_{p{\rm-var},[s,t]}$ and apply \eqref{YL0} with $K$ in \eqref{constK} to derive
\begin{eqnarray*}
\|y_t - y_s \| &\leq& \int_s^t (L \|y_u\| + \|f(0)\|)du + \ltn x \rtn_{p{\rm -var},[s,t]}\Big(\|g(y_s)\| + K C_g \ltn y \rtn_{p{\rm -var},[s,t]}\Big)\\ 
\text{so that}\quad \ltn y \rtn_{p{\rm -var},[s,t]} &\leq& \int_s^t L \ltn y\rtn_{p{\rm -var},[s,u]} du + (\|f(0)\|+ L \|y_s\|)(t-s) \\
&&+ \ltn x \rtn_{p{\rm -var},[s,t]}\Big(\|g(y_s)\| + (K+1) C_g \ltn y \rtn_{p{\rm -var},[s,t]}\Big).
\end{eqnarray*}
As a result, 
\begin{equation}\label{ypvar}
\ltn y \rtn_{p{\rm -var},[s,t]} \leq \int_s^t 2 L \ltn y\rtn_{p{\rm -var},[s,u]} du + 2(\|f(0)\|+ L \|y_s\|)(t-s) +2 \ltn x\rtn_{p{\rm -var},[s,t]} \|g(y_s)\|
\end{equation}
whenever $(K+1)C_g \ltn x\rtn_{p{\rm -var},[s,t]} \leq \frac{1}{2}$. Applying the continuous Gronwall Lemma \cite[Lemma 6.1, p\ 89]{amann} for $\ltn y \rtn_{p{\rm -var},[s,t]}$ yields
\allowdisplaybreaks
\begin{eqnarray}\label{ypvarest}
\ltn y \rtn_{p{\rm -var},[s,t]} &\leq&  2(\|f(0)\|+ L \|y_s\|)(t-s) + 2 \ltn x\rtn_{p{\rm-var},[s,t]}\|g(y_s)\| \notag\\
&&+ \int_s^t 2L e^{2L (t-u)} \Big[2(\|f(0)\|+ L \|y_s\|)(u-s) + 2 \ltn x\rtn_{p{\rm-var},[s,u]}\|g(y_s)\| \Big] du \notag\\
&\leq& \Big(M_0+ \frac{K+2}{K+1}\|y_s\|\Big)e^{2L(t-s)} - \|y_s\|
\end{eqnarray}
 whenever $\ltn x\rtn_{p{\rm -var},[s,t]} \leq \gamma= \frac{1}{2(K+1)C_g}$. Now construct the sequence of greedy times $\{\tau_k=\tau_k(\gamma)\}_{k \in \N}$ on interval $[a,b]$ as in \eqref{greedytime}, it follows from induction that
 \begin{eqnarray*}
 	&& \|y_{\tau_{k+1}} \| \leq \|y\|_{\infty,[\tau_k,\tau_{k+1}]}\leq \| y \|_{p{\rm -var},[\tau_k,\tau_{k+1}]}
	 \leq \Big(e^\alpha\|y_{\tau_k}\|+M_0 \Big) e^{2L(\tau_{k+1}-\tau_k)} \\
 	&&\leq \Big[ \|y_a\| +M_0(k+1)\Big] e^{\alpha(k+1)+2L (\tau_{k+1}-\tau_0)}, \quad \forall k = 0, \ldots, N_{[a,b]}(x)-1, 
 \end{eqnarray*}
 which proves \eqref{estx} since $\tau_{N_{[a,b]}(x)} = b$. On the other hand, it follows from inequality of $p$-variation seminorm in Lemma \ref{additive} and \eqref{ypvarest} that for all $k = 0, \ldots, N_{[a,b]}(x)-1$, 
  \begin{eqnarray*}
 	\ltn y \rtn_{p{\rm - var},[a,b]} 
  &\leq& N^{\frac{p-1}{p}}_{[a,b]}(x)\sum_{k = 0}^{N_{[a,b]}(x)-1} \Big\{ \|y_{\tau_k}\| \big(e^{\alpha+2L(\tau_{k+1}-\tau_k)} -1\big)+ M_0 e^{2L(\tau_{k+1}-\tau_k)} \Big\}\\
 	&\leq&  N^{\frac{p-1}{p}}_{[a,b]}(x) \Big[\|y_a\| + M_0N_{[a,b]}(x)\Big]e^{\alpha N_{[a,b]}(x)+2L(b-a)} - \|y_a\|
 \end{eqnarray*}
 which proves \eqref{estx2}. 
\end{proof}
By the same arguments, we can prove the following results.
\begin{corollary}\label{colnew}
	If in addition $g$ is bounded by $\|g\|_\infty < \infty$, then 
	\begin{equation}\label{esty2}
	\| y \|_{p{\rm -var},[a,b]} \leq \Big[\|y_a\| + \Big(\frac{\|f(0)\|}{L} \vee 2\|g\|_\infty\Big) (1+  \ltn x\rtn_{p{\rm-var},[a,b]})N_{[a,b]}(x)\Big] e^{2 L (b-a)}N^{\frac{p-1}{p}}_{[a,b]}(x),
	\end{equation}
in which $a\vee b := \max\{a,b\}$.
\end{corollary}
\begin{corollary}\label{coryest}
	The following estimate holds
	\begin{eqnarray}\label{yqvar}
	\| y \|_{p{\rm -var},[a,b]} &\leq& \Big[\|y_a\| +\left(\frac{\|f(0)\|}{L}\vee 2\|g(0)\|\right) (1+  \ltn x\rtn_{p{\rm-var},[a,b]})N_{[a,b]}(x)\Big]\times \notag\\
	&& \times e^{\alpha N_{[a,b]}(x)+2L(b-a) }N^{\frac{p-1}{p}}_{[a,b]}(x).
	\end{eqnarray}
\end{corollary}

\medskip

The lemma below is useful in evaluating the difference of two solutions of equation \eqref{fSDE0}. 	The proof is similar to \cite[Lemma 3.1]{congduchong17} and will be omitted here.

\begin{lemma}\label{Q}
Let $y^1,y^2$ be two solution of \eqref{fSDE0}. Assign 
\[
Q(t)=Q(t,y^1,y^2) := g(y^1_t)-g(y^2_t), t\geq 0
\]
where $g$ satisfies (${\textbf H}_2$).

$(i)$ If in addition, $D_g$  is of Lipschitz continuity with Lipschitz constant $C'_g$, then
\begin{equation}\label{Q3}
	\ltn Q\rtn_{p{\rm-var},[u,v]}\leq C_g\ltn y^1-y^2\rtn_{p{\rm-var},[u,v]} + C'_g \|y^1-y^2\|_{\infty,[u,v]} \ltn y^1\rtn_{p{\rm-var},[u,v]}.
	\end{equation}	
	
$(ii)$ If g is a linear map, then $\ltn Q\rtn_{p{\rm-var},[u,v]}\leq C_g\ltn y^1-y^2\rtn_{p{\rm-var},[u,v]}$.
\end{lemma}
Thanks to Lemma \ref{Q}, the difference of two solutions of \eqref{fSDE0} can be estimated in $p$-var norm as follows. 
\begin{corollary}\label{2sol1}
 Let $y^1,y^2$ be two solutions of \eqref{fSDE0} and assign $z_t=y^2_t-y^1_t$ for all $t\geq 0$.

$(i)$ If $D_g$  is of Lipschitz continuity with Lipschitz constant $C'_g$ then
	\begin{equation}\label{z.gen}
	\| z \|_{p{\rm -var},[a,b]} \leq \|z_a\|(N^\prime_{[a,b]}(x))^{\frac{p-1}{p}} 2^{N^\prime_{[a,b]}(x)}e^{2L (b-a)} , \ \forall a\leq b
	\end{equation}
in which
\begin{eqnarray}\label{Nprime}
	N^\prime_{[a,b]}(x) &\leq& 1+[2(K+1)(C_g\vee C'_g)]^p \ltn x \rtn_{p{\rm -var},[a,b]}^p(1+ \ltn y^1 \rtn_{p{\rm -var},[a,b]})^p.
	\end{eqnarray}
	
$(ii)$ If in addition g is a linear map then
\begin{eqnarray}\label{z.lin}
\| z \|_{p{\rm -var},[a,b]} &\leq&\|z_a\|e^{\alpha N_{[a,b]}(x) +2L (b-a)}. 
\end{eqnarray}

\end{corollary}
\begin{proof} 
The proof use similar arguments to the proof of Theorem \ref{YDE}, thus it will be omitted here. The the readers are referred to \cite[Proposition 1]{lejay}, \cite[Theorem 3.9]{congduchong17} for similar versions.	
\end{proof}

\begin{remark}\label{ysummary}
It follows from \eqref{Nest} that
	\begin{eqnarray*}
		N_{[a,b]}^\frac{p-1}{p}(x) &\leq& \Big(1 + [2(K+1)C_g]^p \ltn x \rtn^p_{p{\rm -var},[a,b]}\Big)^\frac{p-1}{p} \leq 1 +[2(K+1)C_g]^{p-1} \ltn x \rtn^{p-1}_{p{\rm -var},[a,b]},\\
		N_{[a,b]}^{\frac{2p-1}{p}}(x) &\leq& \Big(1 + [2(K+1)C_g]^p \ltn x \rtn^p_{p{\rm -var},[a,b]}\Big)^\frac{2p-1}{p} \leq 2^{\frac{p-1}{p}} \Big(1+[2(K+1)C_g]^{2p-1} \ltn x \rtn^{2p-1}_{p{\rm -var},[a,b]}\Big).
	\end{eqnarray*}
As a result, the norm estimates \eqref{estx2}, \eqref{esty2}, \eqref{yqvar}  have the same form
\begin{eqnarray}\label{genform}
	\| y \|_{p{\rm -var},[a,b]} &\leq& \|y_a\| \Lambda_1(x,[a,b])+ \Lambda_2(x,[a,b])
	\end{eqnarray}
in which  $\Lambda_i(x,[a,b])$ are functions of $\ltn x\rtn_{p{\rm-var},[a,b]}$. Similarly, \eqref{z.gen} (for a fixed solution $y^1$) and \eqref{z.lin} can also be rewritten in the form \eqref{genform} with $\Lambda_2 \equiv 0$.
\end{remark}

In the following, let $y$ be a solution of \eqref{fSDE0} on $[a,b]\subset \R^+$ and $\mu$ be the solution of the corresponding deterministic system, i.e
\begin{equation}\label{mu.equ}
\dot{\mu_t} = A\mu_t + f(\mu_t), t\in [a,b]
\end{equation}
with the same initial condition $\mu_a = y_a$. Assign $h_t := y_t - \mu_t$. The following result, which is used in studying singleton attractors in Theorem \ref{gbounded}, estimates the norms of $h$ with the initial condition $\|y_a\|$, up to a fractional order.
\begin{corollary}\label{h}
Assume that $g$ is bounded. Then for a fixed constant $\beta = \frac{1}{p} \in (\frac{1}{2},1)$, there exists for each interval $[a,b]$ a constant $D$ depending on $b-a$ such that
\begin{eqnarray}
\|h\|_{\infty,[a,b]} &\leq&  D\Big( \|y_a\|^\beta + 1 \Big)\ltn x\rtn_{p{\rm-var},[a,b]} N_{[a,b]}(x),\label{hinf} \\
\|h\|_{p{\rm-var},[a,b]} &\leq&  D\Big( \|y_a\|^\beta + 1 \Big)\ltn x\rtn_{p{\rm-var},[a,b]} N^{\frac{2p-1}{p}}_{[a,b]}(x).\label{hpvar}
\end{eqnarray}
\end{corollary}
\begin{proof}
	The proof follows similar steps to \cite[Proposition 4.6]{hairer07} with only small modifications in estimates for $p$ - variation norms and in usage of the continuous Gronwall lemma. To sketch out the proof, we first observe from \eqref{mu.equ} with $ r= b-a$ that
	\begin{equation} \label{muest1}
	\|\mu_t-\mu_s\| \leq \int_s^t (\|f(0)\| + L \|\mu_u\|)du \leq e^{Lr} (L\|\mu_a\|+ \|f(0)\| ) (t-s),\; a\leq s\leq t\leq b.
	\end{equation}
 Next, it follows  from ${\textbf H}_2$ and the boundedness of $g$ by $\|g\|_\infty$ that 
	\begin{equation}\label{hest1}
	\|h_{t}-h_s\| \leq \int_s^t L\|h_u\| du+  \|g\|_{\infty} \ltn x\rtn_{p{\rm-var},[s,t]} + K \ltn x \rtn_{p{\rm-var},[s,t]} \ltn g(h+\mu ) \rtn_{p{\rm-var},[s,t]}.
	\end{equation}
Observe that due to the boundedness of $g$,
	\begin{eqnarray*}\label{muest2}
	\|g(h_t+\mu_t) - g(h_s+\mu_s )\| 
	&\leq& C_g \|h_{t}-h_s\| + (2\|g\|_{\infty}\vee C_g) \|\mu_{t}-\mu_s\|^\beta \\\notag
	&\leq& C_g\|h_{t}-h_s\| +D(1+\|\mu_a\|^\beta) (t-s)^\beta,\qquad \forall a \leq s < t \leq b
	\end{eqnarray*}
	where the last estimate is due to \eqref{muest1} and $D$ is a generic constant depending on $b-a$. This leads to
	\begin{equation}\label{muest3}
	\ltn g(h+\mu)\rtn_{p{\rm-var},[s,t]} \leq C_g \ltn h \rtn_{p{\rm-var},[s,t]} +D(1+\|\mu_a\|^\beta)(t-s)^\beta. 
	\end{equation}
	Replacing \eqref{muest3} into \eqref{hest1} yields
	\begin{eqnarray*}
	\ltn h \rtn_{p{\rm-var},[s,t]} &\leq&  \int_s^t L	\ltn h \rtn_{p{\rm-var},[s,u]} du+  L\|h_s\|(t-s)+  \Big[\|g\|_{\infty} \vee D(1+\|\mu_a\|^\beta)  \Big]\ltn x\rtn_{p{\rm-var},[s,t]} +\\
	&& + KC_g \ltn x \rtn_{p{\rm-var},[s,t]} \ltn h \rtn_{p{\rm-var},[s,t]} 
	\end{eqnarray*}
	which is similar to \eqref{ypvar}. Using similar arguments to the proof of Theorem \ref{YDE} and taking into account \eqref{esty2}, we conclude that 
	\[
	\|h\|_{\infty,[a,b]} \leq e^{2Lr} \Big[ \|h_a\| + D(1+\|\mu_a\|^\beta)  \ltn x\rtn_{p{\rm-var},[a,b]}  N_{[a,b]}(x)\Big],
	\]
 for a generic constant $D$. Finally, \eqref{hinf} is derived since $h_a=0$. The estimate \eqref{hpvar} is obtained similarly.
\end{proof}


\section{Random attractors}

\subsection{Generation of random dynamical systems}

In this subsection we would like to present the generation of a random dynamical system from Young equation \eqref{stochYDE}. Let $(\Omega,\mathcal{F},\mP)$ be a probability space equipped with a so-called {\it metric dynamical system} $\theta$, which is a measurable mapping $\theta: \R \times \Omega \to \Omega$ such that $\theta_t:\Omega\to\Omega$ is $\mP-$ preserving, i.e $\mP(B) = \mP(\theta^{-1}_t(B))$ for all $B\in \mathcal{F}, t\in \R$, and $\theta_{t+s} = \theta_t \circ \theta_s$ for all $t,s \in \R$. A continuous {\it random dynamical system} $\varphi: \R \times \Omega \times \R^d \to \R^d$, $(t,\omega,y_0)\mapsto \varphi(t,\omega)y_0$ is then defined as a measurable mapping which is also continuous in $t$ and $y_0$ such that the cocycle property
\begin{equation}\label{cocycle}
\varphi(t+s,\omega)y_0 =\varphi(t,\theta_s \omega) \circ \varphi(s,\omega)y_0,\quad \forall t,s \in \R, \omega\in \Omega, y_0 \in \R^d
\end{equation}
is satisfied \cite{arnold}.\\ 
In our scenario, denote by $\cC^{0,p-\rm{var}}([a,b],\R^m)$ the closure of $\cC^{\infty}([a,b],\R^m)$ in $\cC^{p-\rm{var}}([a,b],\R^m)$, and by $\cC^{0,p-\rm{var}}(\R,\R^m)$ the space of all $x: \R\to \R^m$ such that $x|_I \in \cC^{0,p-\rm{var}}(I, \R^m)$ for each compact interval $I\subset\R$. Then equip $\cC^{0,p-\rm{var}}(\R,\R^m)$ with the compact open topology given by the $p-$variation norm, i.e  the topology generated by the metric:
\[
d_p(x_1,x_2): = \sum_{k\geq 1} \frac{1}{2^k} (\|x_1-x_2\|_{p{\rm-var},[-k,k]}\wedge 1).
\]
Assign
\[
\Omega:=\cC^{0,p-\rm{var}}_0(\R,\R^m):= \{x\in \cC^{0,p-\rm{var}}(\R,\R^m)|\; x_0=0\},
\]
and equip with the Borel $\sigma -$ algebra $\mathcal{F}$. Note that for $x\in \cC^{0,p-\rm{var}}_0(\R,\R^m)$, $\ltn x\rtn_{p{\rm-var},I} $ and $\|x\|_{p{\rm-var},I}$ are equivalent norms for every compact interval $I$ containing $0$. \\
To equip this measurable space $(\Omega,\cF)$ with a metric dynamical system, consider a stochastic process $\bar{Z}$ defined on a probability space $(\bar{\Omega},\bar{\mathcal{F}},\bar{\bP})$  with realizations in $(\cC^{0,p-\rm{var}}_0(\R,\R^m), \mathcal{F})$. Assume further that $\bar{Z}$ has stationary increments.	Denote by $\theta$ the {\it Wiener shift}
\[
(\theta_t x)_\cdot = x_{t+\cdot} - x_t,\forall t\in \R, x\in \cC^{0,p-\rm{var}}_0(\R,\R^m).
\]
It is easy to check that $\theta$ forms a continuous (and thus measurable) dynamical system $(\theta_t)_{t\in \R}$ on $(\cC^{0,p-\rm{var}}_0(\R,\R), \mathcal{F})$. Moreover, the Young integral satisfies the shift property with respect to $\theta$, i.e.
\begin{equation}\label{shift}
\int_a^b y_u dx_u = \int_{a-r}^{b-r} y_{r+u} d(\theta_r x)_u 
\end{equation}  
(see details in \cite[p.\ 1941]{congduchong17}). It follows from \cite[Theorem 5]{BRSch17} that, there exists a probability $\bP$ on $(\Omega, \mathcal{F}) = (\cC^{0,p-\rm{var}}_0(\R,\R^m), \mathcal{F})$ that is invariant under $\theta$, and the so-called {\it diagonal process}  $Z: \R \times \Omega \to \R^m, Z(t,x) = x_t$ for all $t\in \R, x \in \Omega$, such that $Z$ has the same law with $\bar{Z}$ and satisfies the {\it helix property}:
\[
Z_{t+s}(x) = Z_s(x) + Z_t(\theta_sx), \forall  x \in \Omega, t,s\in \R.
\]
Such stochastic process $Z$ has also stationary increments and almost all of its realizations belongs to $\cC^{0,p-\rm{var}}_0(\R,\R^m)$. It is important to note that the existence of $\bar{Z}$ is necessary to construct the diagonal process $Z$. We assume additionally that  $(\Omega, \mathcal{F}, \bP,\theta) $  is ergodic.\\
When dealing with fractional Brownian motion \cite{mandelbrot}, we can start with the space $\cC_0(\R,\R^m)$ of continuous functions on $\R$ vanishing at zero, with the Borel $\sigma-$algebra $\cF$, and the Wiener shift and the Wiener probability $\mP$, and then follow \cite[Theorem 1]{GASch} to construct an invariant probability measure $\mP^H = B^H \mP$ on the subspace $\cC^\nu$ such that $B^H \circ \theta = \theta \circ B^H$. It can be proved that $\theta$ is ergodic (see \cite{GASch}).\\
Under this circumstance, if we assume further that \eqref{Gamma} is satisfied,
then it follows from  Birkhorff ergodic theorem that
\begin{equation}\label{gamma}
\Gamma(x,p) := \limsup \limits_{n \to \infty} \Big(\frac{1}{n}\sum_{k=1}^{n}  \ltn \theta_{-k}x \rtn^p_{p{\rm -var},[-1,1]}\Big)^{\frac{1}{p}} = \Gamma(p)
\end{equation}	
for almost all realizations $x_t = Z_t(x)$ of $Z$. Particularly, in case $Z=B^H = (B_1^{H}, \dots, B_m^H)$ where $B_i^H$ are scalar fractional Brownian motions (not necessarily independent), we can apply Lemma 2.1 in \cite{congduchong17} with the estimate in \cite[Lemma 4.1 (iii), p.14]{duchongcong18} to obtain that $\Gamma(p) < \infty$.
\begin{proposition}
	The system \eqref{stochYDE}	generates a continuous random dynamical system.
\end{proposition}
\begin{proof}
	The proof follows directly from \cite{BRSch17} and \cite[Section 4.2]{congduchong17}. The generated random dynamical system is defined by $\varphi(t,x)y_0 := y(t,x,y_0)$, for $t\geq 0, x\in \Omega$, which is the pathwise solution of \eqref{stochYDE} w.r.t. the starting point $y_0$ at time $0$. 
\end{proof}

\subsection{Existence of pullback attractors}
Given a random dynamical system $\varphi$ on $\R^d$, we follow \cite{crauelkloeden}, \cite[Chapter 9]{arnold} to present the notion of random pullback attractor. Recall that a set $\hat{M} :=
\{M(x)\}_{x \in \Omega}$ a {\it random set}, if $y
\mapsto d(y|M(x))$ is $\cF$-measurable for each $y \in \R^d$, where $d(E|F) = \sup\{\inf\{d(y, z)|z \in F\} | y \in E\}$  for $E,F$ are nonempty subset of $\R^d$ and $d(y|E) = d(\{y\}|E)$.  
An {\it universe} $\cD$ is a family of random sets
which is closed w.r.t. inclusions (i.e. if $\hat{D}_1 \in \cD$ and
$\hat{D}_2 \subset \hat{D}_1$ then $\hat{D}_2 \in \cD$). \\
In our setting, we define the universe $\cD$ to be a family of {\it tempered} random sets $D(x)$, which means the following: A random variable $\rho(x) >0$ is called {\it tempered} if it satisfies
\begin{equation}\label{tempered}
\lim \limits_{t \to \pm \infty} \frac{1}{t} \log^+ \rho(\theta_{t}x) =0,\quad \text{a.s.}
\end{equation}
(see e.g. \cite[pp. 164, 386]{arnold}) which, by \cite[p. 220]{ImkSchm01}), is equivalent to the sub-exponential growth
\[
\lim \limits_{t \to \pm \infty} e^{-c |t|} \rho(\theta_{t}x) =0\quad \text{a.s.}\quad \forall c >0.
\]
A random set $D(x)$ is called {\it tempered} if it is contained in a ball $B(0,\rho(x))$ a.s., where the radius $\rho(x)$ is a tempered random variable.\\
 A random subset $A$ is called invariant, if $\varphi(t,x)A(x) = A(\theta_t x)$ for all $t\in \R,\; x\in\Omega.$ An invariant random compact set $\mathcal{A}  \in \cD$ is called a {\it pullback random attractor} in $\cD$, if $\mathcal{A} $ attracts
any closed random set $\hat{D} \in \cD$ in the pullback sense,
i.e.
\begin{equation}\label{pullback}
\lim \limits_{t \to \infty} d(\varphi(t,\theta_{-t}x)
\hat{D}(\theta_{-t}x)| \mathcal{A} (x)) = 0.
\end{equation}
$\mathcal{A} $ is called a {\it forward random attractor} in $\cD$, if $\mathcal{A} $ is invariant and attracts
any closed random set $\hat{D} \in \cD$ in the forward sense,
i.e.
\begin{equation}\label{forward}
\lim \limits_{t \to \infty} d(\varphi(t,x)
\hat{D}(x)| \mathcal{A} (\theta_{t}x)) = 0.
\end{equation}
The existence of a random pullback attractor follows from the existence of a random pullback absorbing set (see \cite[Theorem 3]{crauelkloeden}). A random set $\mathcal{B}  \in \cD$ is called {\it pullback
	absorbing} in a universe $\cD$ if $\mathcal{B} $ absorbs all sets in
$\cD$, i.e. for any $\hat{D} \in \cD$, there exists a time $t_0 =
t_0(x,\hat{D})$ such that
\begin{equation}\label{absorb}
\varphi(t,\theta_{-t}x)\hat{D}(\theta_{-t}x) \subset
\mathcal{B} (x), \ \textup{for all}\  t\geq t_0.
\end{equation}
Given a universe $\cD$ and a random compact
pullback absorbing set $\mathcal{B} \in \cD$, there exists a unique random pullback attractor
 in $\cD$, given by
\begin{equation}\label{at}
\mathcal{A}(x) = \cap_{s \geq 0} \overline{\cup_{t\geq s} \varphi(t,\theta_{-t}x)\mathcal{B}(\theta_{-t}x)}. 
\end{equation}
Thanks to the rule of integration by parts for Young integral, the "variation of constants" formula for Young differential equations holds (see e.g. \cite{zahle} or \cite{ducGANSch18}), so that $y_t$ satisfies
\begin{equation}\label{variation}
y_t =\Phi (t-a)y_a + \int_a^t\Phi (t-s)f(y_s) ds + \int_a^t\Phi (t-s)g(y_s) d x_s,\quad \forall t\geq a.
\end{equation} 
We need the following auxiliary results. 
\begin{proposition}\label{YDEg}
	Given \eqref{estphi1} and \eqref{estphi2}, the following estimate holds: for any $0\leq a<b\leq c$
	\begin{eqnarray}\label{int2}
	\Big\|\int_a^{b}\Phi (c-s)g( y_s) d  x_s\Big\| &\leq& KC_A \Big[1+|A|(b-a)\Big]\ltn  x\rtn_{p{\rm -var},[a,b]}e^{-\lambda_A(c-b)} \Big[C_g \| y\|_{p{\rm -var},[a,b]}+\|g(0)\|\Big]. \notag\\ 
	\end{eqnarray}
\end{proposition}

\begin{proof}
	The proof follows directly from \eqref{estphi1} and \eqref{estphi2} as follows
	\begin{eqnarray*}
		&&\Big\|\int_a^b\Phi (c-s)g(y_s) d  x_s\Big\|\notag\\
		&\leq & \ltn  x\rtn_{p\rm{-var},[a,b]}\left( \|\Phi(c-a) g( y_a)\| + K\ltn \Phi(c-\cdot)g(y_\cdot)\rtn_{p{\rm-var},[a,b]}  \right)\notag\\
		&\leq&  \ltn  x\rtn_{p{\rm-var},[a,b]}\Big\{ \|\Phi(c-a)\| \| g( y_a)\| \Big.\notag\\
		&&+\Big.  K\left( \ltn \Phi(c-\cdot)\rtn_{p{\rm -var},[a,b] }\| g(y)\|_{\infty,[a,b]}  +   \| \Phi(c-\cdot)\|_{\infty,[a,b] }\ltn g(y)\rtn_{p{\rm -var},[a,b]}\right) \Big\}\notag\\
		&\leq&K C_A \ltn  x\rtn_{p{\rm -var},[a,b]} e^{-\lambda_A(c-b)}\times\notag\\
		&&\times\Big[C_g\| y_a\|+\|g(0)\| +|A|(b-a) \big(C_g\| y\|_{\infty,[a,b]} +\|g(0)\|\big)+C_g\ltn y\rtn_{p{\rm -var},[a,b]}\Big]\notag\\
		&\leq& KC_A\Big[1+|A|(b-a)\Big]\ltn  x\rtn_{p{\rm -var},[a,b]}e^{-\lambda_A(c-b)} \Big[C_g \| y\|_{p{\rm -var},[a,b]}+\|g(0)\|\Big] .
	\end{eqnarray*}
\end{proof}	
The following lemma is the crucial technique of this paper.
\begin{lemma}\label{ytest}
	Assume that $y_t$ satisfies \eqref{variation}. Then for any $n\geq 0$, 
	\begin{eqnarray}\label{yt}
	\|y_t\|e^{\lambda t}
	&\leq& C_A \|y_0\| + \frac{C_A}{\lambda_A - L_f} \|f(0)\| \Big(e^{\lambda t}-1\Big)\\
	&&+ \sum_{k=0}^{n} e^{\lambda_A}KC_A(1+|A|)\ltn  x\rtn_{p{\rm -var},\Delta_k}e^{\lambda k} \Big[C_g \| y\|_{p{\rm -var},\Delta_k}+\|g(0)\|\Big],  \forall t\in \Delta_n, \notag
	\end{eqnarray}
	where $\Delta_k:=[k,k+1]$, $L_f$ and $\lambda$ are defined in \eqref{L}.
\end{lemma}
\begin{proof}
	First, for any $t\in [n,n+1)$, it follows from \eqref{estphi1} and the global Lipschitz continuity of $f$ that
	\begin{eqnarray*}
		\|y_t\| &\leq& \|\Phi (t)y_0\| + \int_0^t \|\Phi (t-s)f(y_s)\| ds + \Big\|\int_0^t\Phi (t-s)g(y_s) d x_s \Big\| \\
		&\leq& C_A e^{-\lambda_A t} \|y_0\| + \int_0^t C_A e^{-\lambda_A (t-s)} \Big(C_f \|y_s\| + \|f(0)\|\Big) ds + \Big\|\int_0^t\Phi (t-s)g(y_s) d x_s \Big\|\\
		&\leq& C_A e^{-\lambda_A t} \|y_0\| + \frac{C_A}{\lambda_A} \|f(0)\| (1- e^{-\lambda_At})  +C_AC_f \int_{0}^t e^{-\lambda_A(t-s)}  \|y_s\|ds + \beta_t,
	\end{eqnarray*}
	where $\beta_t := \Big\|\int_0^t\Phi (t-s)g(y_s) d x_s \Big\|$. Multiplying both sides of the above inequality with $e^{\lambda_A t}$ yields
	\[
	\|y_t\|e^{\lambda_A t} \leq C_A \|y_0\| + \frac{C_A}{\lambda_A} \|f(0)\| (e^{\lambda_At}-1) + \beta_t e^{\lambda_A t} +C_AC_f  \int_{0}^t  e^{\lambda_A s} \|y_s\|ds.
	\]
	By applying the continuous Gronwall Lemma \cite[Lemma 6.1, p\ 89]{amann}, we obtain
	\begin{eqnarray*}
		\|y_t\|e^{\lambda_A t} &\leq& C_A \|y_0\| + \frac{C_A}{\lambda_A} \|f(0)\| (e^{\lambda_At}-1) + \beta_t e^{\lambda_A t} \\
		&&+ \int_0^t L_f e^{L_f(t-s)} \Big[C_A \|y_0\| + \frac{C_A}{\lambda_A} \|f(0)\| (e^{\lambda_A s}-1) + \beta_s e^{\lambda_A s}\Big]ds.
	\end{eqnarray*}
	Once again, multiplying both sides of the above inequality with $e^{-L_f t}$ yields
	\begin{eqnarray}\label{yaux}
	\|y_t\|e^{\lambda t} &\leq& C_A \|y_0\| e^{-L_f t} + \frac{C_A}{\lambda_A} \|f(0)\| \Big(e^{\lambda t}-e^{-L_ft}\Big) + \beta_t e^{\lambda t} \notag\\
	&&+ \int_0^t L_f e^{-L_fs} \Big[C_A \|y_0\| + \frac{C_A}{\lambda_A} \|f(0)\| (e^{\lambda_A s}-1) + \beta_s e^{\lambda_A s}\Big]ds \notag\\
	&\leq& C_A \|y_0\| + \frac{C_A}{\lambda} \|f(0)\| \Big(e^{\lambda t}-1\Big)+ \beta_t e^{\lambda t} + \int_0^t L_f \beta_s e^{\lambda s}  ds.
	\end{eqnarray}
Next, assign $p([a,b]) := KC_A \Big[1+|A|(b-a)\Big]\ltn  x\rtn_{p{\rm -var},[a,b]}\Big[C_g \| y\|_{p{\rm -var},[a,b]}+\|g(0)\|\Big]$, and apply \eqref{int2} in Proposition \ref{YDEg}, it follows that for all $s\leq t$
	\allowdisplaybreaks
	\begin{eqnarray}\label{int210}
	\beta_s e^{\lambda s} 
	&=& e^{\lambda s} \Big\|\int_0^s\Phi (s-u)g(y_u) d x_u \Big\| \notag\\
	&\leq& e^{\lambda s} \sum_{k=0}^{\lfloor s \rfloor -1} \Big\|\int_{\Delta_k}\Phi (s-u)g(y_u) d x_u \Big\| +e^{\lambda s} \Big\|\int_{\lfloor s \rfloor}^{s}\Phi (s-u)g(y_u) d x_u \Big\| \notag\\
	&\leq& e^{\lambda s} \sum_{k=0}^{\lfloor  s \rfloor-1} e^{-\lambda_A(s-k-1)} p(\Delta_k) + e^{\lambda s} p ([\lfloor s\rfloor,s])\notag\\
	&\leq& \sum_{k=0}^{\lfloor s\rfloor} e^{\lambda s} e^{-\lambda_A(s-k-1)} p(\Delta_k) 
	\leq \sum_{k=0}^{\lfloor s\rfloor} e^{\lambda_A} e^{\lambda k} e^{-L_f(s-k)} p(\Delta_k). 
	\end{eqnarray}
	By replacing \eqref{int210} into \eqref{yaux} we obtain
	\allowdisplaybreaks
	\begin{eqnarray}\label{yn}
\|y_t\|e^{\lambda  t} 
	&\leq& C_A \|y_0\| + \frac{C_A}{\lambda} \|f(0)\| \Big(e^{\lambda t}-1\Big) \notag\\
	&&+  \sum_{k=0}^{n} e^{\lambda_A} e^{\lambda k} e^{-L_f(t-k)} p(\Delta_k) + L_f \int_0^t \sum_{k=0}^{\lfloor s \rfloor} e^{\lambda_A} e^{\lambda k} e^{-L_f(s-k)} p(\Delta_k) ds\notag\\
	&\leq& C_A \|y_0\| + \frac{C_A}{\lambda} \|f(0)\| \Big(e^{\lambda t}-1\Big)+ \sum_{k=0}^{n} e^{\lambda_A } e^{\lambda k} p(\Delta_k) \Big(e^{-L_f(t-k)}  + \int_{k}^t L_f e^{-L_f(s-k)}ds\Big) \notag\\
	&\leq& C_A \|y_0\| + \frac{C_A}{\lambda} \|f(0)\| \Big(e^{\lambda t}-1\Big) + \sum_{k=0}^{n} e^{\lambda_A } e^{\lambda k} p(\Delta_k),\quad \forall t \in [n, n+1)\notag
	\end{eqnarray}
	where we use the fact that 
	\begin{equation}\label{intequal}
	e^{-L_f(t-a)}  + \int_{a}^t L_f e^{-L_f(s-a)}ds = 1,\quad \forall t \geq a.
	\end{equation}
	The continuity of $y$ at $t= n+1$ then proves \eqref{yt}.\\
\end{proof}

We now formulate the first main result of the paper.

\begin{theorem}\label{attractor}
	Under the assumptions ${\textbf H}_1-{\textbf H}_3$, assume further that  $\lambda_A > C_fC_A$, where $\lambda_A$ and $C_A$ are given from \eqref{estphi1},\eqref{estphi2}. If the criterion \eqref{criterion}
\[
	\lambda_A - C_A C_f> C_A(1+|A|) e^{\lambda_A+2(|A|+C_f)} \Big\{\Big[2(K+1)C_g\Gamma(p)\Big]^p + \Big[2(K+1)C_g\Gamma(p)\Big]\Big\}
\]
	holds, where $\Gamma(p)$ is defined in \eqref{Gamma}, then the generated random dynamical system $\varphi$ of system \eqref{stochYDE} possesses a pullback attractor $\mathcal{A}$. 
\end{theorem}
\begin{proof}		
{\bf Step 1.} To begin,  we rewrite the estimate \eqref{estx} in the short form, using \eqref{genform} in Remark \ref{ysummary}
\begin{equation}\label{Form}
\| y\|_{p{\rm -var},\Delta_k}\leq \|y_k\|\Lambda_1(x,\Delta_k) +M_0\Lambda_2(x,\Delta_k)
\end{equation}
where $\Delta_k=[k,k+1]$, $M_0$ is given by \eqref{M0} and 
\begin{eqnarray}
\Lambda_1(x,[a,b]) &:=& \left(1 +[2(K+1)C_g]^{p-1} \ltn x \rtn^{p-1}_{p{\rm -var},[a,b]} \right) F(x,[a,b]), \notag\\
\Lambda_2(x,[a,b]) &:=&   2^{\frac{p-1}{p}}\left(1 +[2(K+1)C_g]^{2p-1} \ltn x \rtn^{2p-1}_{p{\rm -var},[a,b]} \right) F(x,[a,b]), \label{Lambda}\\
 F(x,[a,b]) &:=& \exp\left\{\alpha  \left(1 +[2(K+1)C_g]^{p} \ltn x \rtn^{p}_{p{\rm -var},[a,b]} \right)+2L (b-a)\right\},\notag
 \end{eqnarray}
for $L, \alpha $  in \eqref{L} and \eqref{constK} respectively. Replacing \eqref{Form} into \eqref{yt} in Lemma \ref{ytest} and using $M_1,M_2$ in \eqref{M1} and \eqref{M2}, we obtain
\allowdisplaybreaks
\begin{eqnarray}\label{xn}
\| y_n\|e^{\lambda n}
&\leq &C_A \|y_0\| +( e^{\lambda  n}-1)\frac{C_A \|f(0)\|}{\lambda }\notag\\
&&+e^{\lambda_A}KC_A(1+|A|) \sum_{k=0}^{n-1} e^{\lambda  k} \ltn  x\rtn_{p{\rm -var},\Delta_k}\Big[C_g\Big(\|y_k\| \Lambda_1(x,\Delta_k) +M_0\Lambda_2(x,\Delta_k)\Big)+\|g(0)\|\Big]   \notag \\
&\leq &C_A \|y_0\| + M_1C_g\sum_{k=0}^{n-1} \ltn  x\rtn_{p{\rm -var},\Delta_k} \Lambda_1(x,\Delta_k) e^{\lambda  k} \|  y_k\| \notag\\
&&  +M_2\sum_{k=0}^{n-1} e^{\lambda k}\left[1+ \ltn  x\rtn_{p{\rm -var},\Delta_k} \Big( 1+\Lambda_2(x,\Delta_k) \Big)\right]
\end{eqnarray}
Assign $a:= C_A\|y_0\|$, $u_k:= e^{\lambda k} \| y_k\|$ and 
	\begin{eqnarray}
	G(x,[a,b]) &:=& \ltn  x\rtn_{p{\rm -var},[a,b]}  \Lambda_1(x,[a,b]),\label{G}\\
	H(x,[a,b]) &:=& 1+ \ltn  x\rtn_{p{\rm -var},[a,b]} \Big( 1+\Lambda_2(x,[a,b]) \Big) \label{H}
\end{eqnarray}
for all $k \geq 0$, where $\Lambda_1,\Lambda_2$ are given by \eqref{Lambda}. Observe that \eqref{xn} has the form
\begin{eqnarray*}
u_n
\leq a + M_1C_g\sum_{k=0}^{n-1} G(x,\Delta_k)  u_k + M_2\sum_{k=0}^{n-1} e^{\lambda k}H(x,\Delta_k).
\end{eqnarray*}
We are in the position to apply the discrete Gronwall lemma \ref{gronwall}, so that
\begin{eqnarray}\label{y}
	\|y_n(x,y_0)\|&\leq &C_A\|y_0\|e^{-\lambda n}  \prod_{k=0}^{n-1} \Big[1+M_1 C_g G(x,\Delta_k)\Big]\notag\\ &&+M_2 \sum_{k=0}^{n-1} e^{-\lambda (n-k)} H(x,\Delta_k) \prod_{j=k+1}^{n-1} \Big[1+M_1C_g G(x,\Delta_j)\Big],\quad \forall n\geq 1.
\end{eqnarray}

{\bf Step 2.} Next, for any $t\in [n,n+1]$, due to \eqref{estx} and \eqref{Nest}, we can write
\allowdisplaybreaks
\begin{eqnarray}
\|y_t(x,y_0)\|&\leq& \|y_n(x,y_0)\|F(x,\Delta_n) + M_0\Lambda_0(x,\Delta_n)  \label{y0}\\
\text{where}\  \Lambda_0(x,[a,b]) &:=& \left(1 +[2(K+1)C_g]^p \ltn x \rtn^{p}_{p{\rm -var},[a,b]} \right) F(x,[a,b]). \label{Lambda0}
\end{eqnarray}
Consequently, replacing $x$ with $\theta_{-t}x$ in \eqref{y0} and using \eqref{y} yields
 \begin{eqnarray}
&&	\|y_t( \theta_{-t}x,y_0(\theta_{-t}x)) \|\notag\\
&\leq &C_A\|y_0(\theta_{-t}x)\|e^{-\lambda n} F(\theta_{-t}x,\Delta_n) \prod_{k=0}^{n-1} \Big[1+M_1C_g G(\theta_{k-t} x,[0,1])\Big] + M_0\Lambda_0(\theta_{-t}x,\Delta_n)  \notag\\
	&&+M_2 \sum_{k=0}^{n-1} e^{-\lambda (n-k)}F(\theta_{-t}x,\Delta_n)H(\theta_{k-t} x, [0,1]) \prod_{j=k+1}^{n-1} \Big[1+M_1C_g G(\theta_{j-t}x,[0,1])\Big]\notag\\
	&\leq& C_A F(x,[-1,1])\|y_0(\theta_{-t}x)\|e^{-\lambda n} \sup_{\epsilon\in[0,1]} \prod_{k=1}^{n} \Big[1+M_1 C_g G(\theta_{-k} x,[-\epsilon,1-\epsilon])\Big]+M_0  \Lambda_0(x,[-1,1]) \notag\\
	&&+M_2F( x,[-1,1]) \sup_{\epsilon\in[0,1]}\sum_{k=1}^{n} e^{-\lambda k}  H(\theta_{-k} x, [-\epsilon,1-\epsilon]) \prod_{j=1}^{k-1} \Big[1+M_1C_g G(\theta_{-j}x,[-\epsilon,1-\epsilon])\Big]\notag\\
	&\leq& C_AF(x,[-1,1])\|y_0(\theta_{-t}x)\|e^{-\lambda n}  \sup_{\epsilon\in[0,1]}  \prod_{k=1}^{n} \Big[1+M_1 C_g G(\theta_{-k} x,[-\epsilon,1-\epsilon])\Big]+M_0  \Lambda_0(x,[-1,1]) \notag\\
	&&+M_2F( x,[-1,1])b(x),\label{yy0}\\
&\text{where}& b(x):= \sup_{\epsilon\in[0,1]} \sum_{k=1}^{\infty}e^{-\lambda k}  H(\theta_{-k}x,[-\epsilon,1-\epsilon])\prod_{j=1}^{k-1} \Big[1+M_1C_gG( \theta_{-j}x,[-\epsilon,1-\epsilon])\Big],	\label{bx}
\end{eqnarray}
($b(x)$ can take value $\infty$). Applying the inequality $\log(1+ae^b)\leq a+b$ for $a,b\geq 0$ and using \eqref{G}, \eqref{Form}, we obtain
	\begin{eqnarray*}
		&& \log \Big(1+ M_1C_g G(x,[-\epsilon,1-\epsilon])\Big) \\
		&\leq&   M_1C_ge^{\alpha+2L} \Big[ \ltn  x\rtn_{p{\rm -var},[-1,1]} + [2(K+1)C_g]^{p-1}\ltn  x\rtn_{p{\rm -var},[-1,1]}^p \Big] +\frac{[2(K+1)C_g]^p \ltn x \rtn^p_{p{\rm -var},[-1,1]}}{K+1} \\
		&\leq& \Big[M_1e^{\alpha+2L}+ 2\Big][2(K+1)]^{p-1}C_g^p  \ltn x \rtn^p_{p{\rm -var},[-1,1]} + M_1e^{\alpha+4L}C_g \ltn  x\rtn_{p{\rm -var},[-1,1]}\\
		&\leq&  C_A e^{\lambda_A+2L}(1+|A|) \Big\{\Big[2(K+1)C_g\ltn  x\rtn_{p{\rm -var},[-1,1]}\Big]^p + \Big[2(K+1)C_g\ltn  x\rtn_{p{\rm -var},[-1,1]}\Big]\Big\}.
	\end{eqnarray*}
	Together with \eqref{gamma} and \eqref{M1}, it follows that for a.s. all $x$,
	\allowdisplaybreaks
	\begin{eqnarray}\label{Gnew}
		&&\limsup \limits_{n \to \infty}\frac{1}{n}\log \left\{\sup_{\eps\in[0,1]}\prod_{k=1}^{n} \Big[1+M_1 C_gG(\theta_{-k}x,[-\epsilon,1-\epsilon])\Big]\right\}\notag\\
		&&\leq C_A e^{\lambda_A+2L}(1+|A|) \Big\{\Big[2(K+1)C_g\Gamma(p)\Big]^p + \Big[2(K+1)C_g\Gamma(p)\Big]\Big\} = \hat{G},
	\end{eqnarray}
	where $ \hat{G}$ is defined in \eqref{Gmu} and is also the right hand side of \eqref{criterion}, and $L$ is defined in \eqref{L}. 
Hence for $t\in \Delta_n$ with $0<\delta< \frac{1}{2}(\lambda-\hat{G})$ and $n \geq n_0$ large enough
	\[
	e^{(-\delta+\hat{G})n}\leq \sup_{\eps\in[0,1]} \prod_{k=0}^{n-1} \Big[1+M_1 C_gG( \theta_{-k}x,[-\eps,1-\eps])\Big]\leq e^{(\delta+\hat{G})n}.
	\]
Starting from any point $y_0(\theta_{-t}x) \in D(\theta_{-t}x)$ which is tempered, there exists, due to \eqref{tempered}, an $n_0$ independent of $y_0$ large enough such that for any $n \geq n_0$ and any $t\in[n,n+1]$ 	
\begin{eqnarray}
\|y_t(\theta_{-t}x,y_0(\theta_{-t}x))\|
&\leq&C_Ae^{\lambda_A}\|y_0(\theta_{-t}x)\|F(x,[-1,1])\exp\left\{-\left(\lambda-\hat{G} -\delta\right)n\right\}\label{yfinal}\\
&&+M_0\Lambda_0(x,[-1,1])+M_2F(x,[-1,1]) b(x) \notag\\
&\leq&1+M_2 F( x,[-1,1])b( x) + M_0\Lambda_0(x,[-1,1]) =: \hat{b}(x)\label{bhat}
\end{eqnarray}
where $F, \Lambda_0$ are given in \eqref{Lambda} and \eqref{Lambda0}. In addition, it follows from \eqref{Nest} and the inequality $\log (1+ ab) \leq \log(1+a) + \log b$ for all $a \geq 0, b\geq 1$, that
\begin{eqnarray}\label{bhat.temp.}
\log \hat{b}(x) &\leq& \log [1+ M_2 F( x,[-1,1])] + \log [1+b(x)] \notag\\
&& + \log \Big\{1+M_0\Big[1 + \Big(2(K+1)C_g\Big)^p \ltn x \rtn^p_{p{\rm -var},[-1,1]}\Big] F(x,[-1,1]) \Big\}\notag\\
&\leq&D+ \log [1+b(x)]+  [2(K+1)C_g]^p \ltn x \rtn^p_{p{\rm -var},[-1,1]} + 2\log F (x,[-1,1])\notag\\
&\leq&  D(1+ \ltn x \rtn^p_{p{\rm -var},[-1,1]})+ \log [1+b(x)]
\end{eqnarray}
where $D$ is a constant.

{\bf Step 3.} Notice that \eqref{gamma} implies $\lim\limits_{n\to\infty} \frac{\ltn \theta_{-n}x\rtn_{p{\rm-var},[-1,1]}}{n}=0$. The proof would be complete if one can prove Proposition \ref{Glim} below, that $b(x)$ is finite and tempered a.s. Indeed, assume that Proposition \ref{Glim} holds, then by applying \cite[Lemma \ 4.1.2]{arnold}, we obtain the temperedness of $\hat{b}(x)$ in the sense of \eqref{tempered}. We conclude that there
exists a compact absorbing set $\mathcal{B}(x) = \bar{B}(0,\hat{b}(x))$ and thus a pullback attractor $\mathcal{A}(x)$ for system \eqref{fSDE0} in the form of \eqref{at}. 
\end{proof}

To complete the proof of Theorem \ref{attractor}, we now formulate and prove Proposition \ref{Glim}.
\begin{proposition}\label{Glim}
	Assume that \eqref{criterion} holds. Then $b(x)$ defined in \eqref{bx} 
\[
	b(x):=\sup_{\epsilon\in[0,1]} \sum_{k=1}^{\infty}e^{-\lambda k}  H(\theta_{-k}x,[-\epsilon,1-\epsilon])\prod_{j=1}^{k-1} \Big[1+M_1C_gG( \theta_{-j}x,[-\epsilon,1-\epsilon])\Big],\quad \forall x\in \Omega
\]
is finite and tempered a.s. (in the sense of \eqref{tempered}). 
\end{proposition}

\begin{proof}
	From the definition of $H$ in \eqref{H}, by similar computations as in \eqref{bhat.temp.} using the integrability of $\ltn x\rtn_{p{\rm-var},[-1,1]}$  it is easy to prove that $\log H(x,[-1,1])$ is integrable, thus
	\begin{eqnarray*}
		 \limsup \limits_{n \to \infty}\frac{\log H(\theta_{-n}x,[-1,1])}{n} = 0\quad \text{a.s.}
	\end{eqnarray*}
	Hence, under condition \eqref{criterion}, there exists for each $0<2\delta < \lambda- \hat{G}$ an $n_0=n_0(\delta, x)$ such that for all $n\geq n_0$,
	\[
	e^{-\delta n}\leq  H(\theta_{-n}x,[-1,1])\leq e^{\delta n}.
	\]
	Consequently,
	\begin{eqnarray*}
		b( x)
		&\leq& \sum_{k=1}^{n_0-1}e^{-\lambda k} H(\theta_{-k}x,[-1,1])\sup_{\eps\in[0,1]} \prod_{j=1}^{k-1} \Big(1+M_1C_g G(\theta_{-j}x,[-\eps,1-\eps])\Big) + \sum_{k=n_0}^{\infty} e^{-(\lambda- 2\delta - \hat{G})k}  \notag\\
		&\leq&\sum_{k=1}^{n_0-1}e^{-\lambda k} H(\theta_{-k}x,[-1,1]) \prod_{j=1}^{k-1} \Big(1+M_1C_g G(\theta_{-j}x,[-1,1])\Big) +  \frac{e^{-(\lambda - 2\delta - \hat{G})n_0}}{1- e^{-(\lambda -2 \delta - \hat{G})}} 
	\end{eqnarray*}
	which is finite a.s. Moreover, for each fixed $x$, $\ltn x\rtn_{p{\rm-var},[s,t]}$ is continuous w.r.t $(s,t)$ on $\{(s,t) \in \R^2| s\leq t\}$ (see \cite[Proposition 5.8,\ p.\ 80]{friz}). Therefore, $G(x,[-\eps,1-\eps])$ and $H(x,[-\eps,1-\eps])$ are continuous functions of $\eps$. Since  $b(x)\leq b^*(x)$, the series 
	$$ \sum_{k=1}^{\infty}e^{-\lambda k}  H(\theta_{-k}x,[-\epsilon,1-\epsilon])\prod_{j=1}^{k-1} \Big[1+M_1C_gG( \theta_{-j}x,[-\epsilon,1-\epsilon])\Big]$$
	 converges uniformly w.r.t. $\epsilon \in [0,1]$, thus the series is also continuous w.r.t. $\eps \in [0,1]$. Hence the supremum in the definition of $b(x)$ in \eqref{bx} can be taken for rational $\epsilon$, which proves $b(x)$ to be a random variable on $\Omega$.
	 
	To prove the temperdness of $b$, observe that for each $t\in [n,n+1]$
	\begin{eqnarray*}
		b(\theta_tx)&=&b(\theta_n\theta_{t-n}x)\\
		&=&\sup_{\epsilon\in[0,1]} \sum_{k=1}^{\infty}e^{-\lambda k}  H(\theta_{-k+n}\theta_{-\epsilon+t-n}x,[0,1])\prod_{j=1}^{k-1} \Big[1+M_1C_gG( \theta_{-j+n}\theta_{-\epsilon+t-n}x,[0,1])\Big]\\
		&\leq &\sup_{\epsilon\in[-1,1]} \sum_{k=1}^{\infty}e^{-\lambda k}  H(\theta_{-k+n}x,[-\epsilon,1-\epsilon])\prod_{j=1}^{k-1} \Big[1+M_1C_gG( \theta_{-j+n}x,[-\epsilon,1-\epsilon])\Big]\\
		&\leq& \max\{b(\theta_nx),b(\theta_{n+1}x)\}.
	\end{eqnarray*}
	For $n> 0$,
	\allowdisplaybreaks
	\begin{eqnarray}\label{b1}
	b(\theta_{-n}x)&=& \sup_{\eps\in[0,1]} \sum_{k=1}^{\infty}e^{-\lambda  k}   H(\theta_{-(k+n)}x,[-\eps,1-\eps])\prod_{j=1}^{k-1} \Big[1+M_1C_gG( \theta_{-(j+n)}x,[-\eps,1-\eps])\Big]\notag \\
	&\leq& e^{\lambda  n} \sup_{\eps\in[0,1]} \sum_{k=1}^{\infty}e^{-\lambda  (n+k)}  H(\theta_{-(k+n)}x,[-\eps,1-\eps])\prod_{j=1}^{n+k-1} \Big[1+M_1C_gG( \theta_{-(j+n)}x,[-\eps,1-\eps])\Big]\notag \\
	&\leq&e^{\lambda  n}  b(x).\notag
	\end{eqnarray}
	Therefore 
	\[
	\displaystyle \limsup \limits_{n\to +\infty} \frac{\log b(\theta_{-n}x)}{n} \leq \lambda <\infty.
	\]
 Applying \cite[Proposition 4.1.3(i), p. 165]{arnold} we conclude that
 \[
	\displaystyle \limsup \limits_{t\to -\infty} \frac{\log b(\theta_tx)}{|t|}=\displaystyle \limsup \limits_{t\to +\infty} \frac{\log b(\theta_{t}x)}{t}=0,
	\]
i.e. $b$ is tempered.
\end{proof}	

\begin{corollary}\label{onepoint}
	Assume that $f(0) = g(0) = 0$ so that $y \equiv 0$ is a solution of \eqref{stochYDE}. Then under the assumptions of Theorem \ref{attractor} with condition \eqref{criterion}, the random attractor $\mathcal{A}(x)$ is the fixed point $ 0$.
\end{corollary}
\begin{proof}
Using \eqref{yfinal} and the fact that $M_0=M_2 =0$ if $f(0) = g(0) =0$, we obtain
\begin{eqnarray}
\|y_t(\theta_{-t}x,y_0)\| \leq C_Ae^{\lambda_A}F(x,[-1,1])\|y_0(\theta_{-t}x)\|\exp\left\{-\left(\lambda-\hat{G} -\delta\right)n\right\}
\end{eqnarray}
for $t\in \Delta_n$.
It follows that all other solutions converge exponentially in the pullback sense to the trivial solution, which plays the role of the global pullback attractor. 
\end{proof}
\begin{remark}\label{comparison}
 In \cite{GAKLBSch2010} and \cite{ducGANSch18} the authors consider a Hilbert space $V$ together with a covariance operator $Q$ on $V$ such that $Q$ is of a trace-class, i.e. for a complete orthonormal basis $(e_i)_{i\in \N}$ of $V$, there exists a sequence of nonnegative numbers $(q_i)_{i \in \N}$  such that $\text{tr}(Q) :=\sum_{i=1}^\infty q_i < \infty$. A $V-$ valued fractional Brownian motion $B^H = \sum_{i=1}^\infty \sqrt{q_i} \beta_i^H e_i$ is then considered, where $(\beta^H_i)_{i \in \N}$ are stochastically independent scalar fractional Brownian motions of the same Hurst exponent $H$. The authors then develop the semigroup method to estimate the H\"older norm of $y$ on intervals $\tau_k, \tau_{k+1}$ where $\tau_k$ is a sequence of stopping times
\[
\tau_0 = 0, \tau_{k+1} - \tau_k + \ltn x \rtn_{\beta,[\tau_k,\tau_{k+1}]} = \mu
\]
for some $\mu \in (0,1)$ and $\beta > \frac{1}{p}$, which leads to the estimate of the exponent as
\begin{equation}
- \Big(\lambda_A - C(C_A,\mu) e^{\lambda_A \mu}\max\{C_f,C_g\}  \frac{n}{\tau_n}\Big)\tau_n,
\end{equation}
where $C(C_A,\mu)$ is a constant depending on $C_A, \mu$. It is then proved that there exists $\liminf \limits_{n \to \infty} \frac{\tau_n}{n} = \frac{1}{d}$, where $d = d(\mu, \text{tr}(Q))$ is a constant depending on the moment of the stochastic noise. As such the exponent is estimated as 
\begin{equation}\label{criterion2}
-\Big(\lambda_A - C(C_A,\mu) e^{\lambda_A \mu}\max\{C_f,C_g\} d(\mu,\text{tr}(Q))\Big). 
\end{equation}
However, it is technically required from the stopping time analysis (see \cite[Section 4]{ducGANSch18}) that the stochastic noise has to be small in the sense that the trace $\text{tr}(Q)=\sum_{i=1}^\infty q_i$ must be controlled as small as possible. In addition, in case the noise is diminished, i.e. $g \equiv 0$, \eqref{criterion2} reduces to a very rough criterion for exponential stability of the ordinary differential equation
\[
C_f \leq \frac{1}{C(C_A,\mu) d(\mu,\text{tr}(Q))} \lambda_A e^{-\lambda_A}.
\]
In contrast, our method uses the greedy time sequence in Theorem \ref{YDE}, so that later we can work with the simpler (regular) discretization scheme without constructing additional stopping time sequence. Also in Lemma \ref{ytest} we apply first the continuous Gronwall lemma in \eqref{yaux} in order to clear the role of the drift coefficient $f$. Then by using \eqref{estx} to give a direct estimate of $y_k$, we are able apply the discrete Gronwall Lemma directly and obtain a very explicit criterion. \\
The left and the right hand sides of criterion \eqref{criterion}
\[
\lambda_A - C_A C_f> C_A(1+|A|) e^{\lambda_A+2(|A|+C_f)} \Big\{\Big[2(K+1)C_g\Gamma(p)\Big]^p + \Big[2(K+1)C_g\Gamma(p)\Big]\Big\}
\]
can be interpreted as, respectively, the decay rate of the drift term and the intensity of the diffusion term, where the term $e^{\lambda_A+4(|A|+C_f)}$ is the unavoidable effect of the discretization scheme. Criterion \eqref{criterion} is therefore a better generalization of the classical criterion for stability of ordinary differential equations, and is satisfied if either $C_g$ or $\Gamma(p)$ is sufficiently small. In particular, when $g \equiv 0$, \eqref{criterion} reduces to $\lambda_A > C_A C_f$, which matches to the classical result.
\end{remark}

\subsection{Singleton attractors}\label{singletonsub}
In the rest of the paper, we would like to study sufficient conditions for the global attractor to consist of only one point, as seen, for instance, in Corrollary \ref{onepoint}. First, the answer is affirmative for $g$ of linear form, as proved in \cite{duchongcong18} for dissipative systems. Here we also present a similar version using the semigroup  method.

To begin, let $y^1,y^2$ be two solutions of \eqref{fSDE0} and assign $z_t=y^2_t-y^1_t$ for all $t\geq 0$. Similar to \eqref{variation}, $z$ satisfies
\begin{equation}\label{variationz}
z_t =\Phi (t-a)z_a + \int_a^t\Phi (t-s) \Big[f(z_s+y^1_s) - f(y^1_s)\Big]ds + \int_a^t\Phi (t-s) Q(s,z_s)d x_s,\quad \forall t\geq a,
\end{equation} 
where $Q(s,z_s) = g(z_s+y^1_s) - g(y^1_s)$. Observe that by similar computations to \eqref{int2}, it is easy to prove that
\begin{eqnarray}\label{estQ}
\Big\| \int_a^b \Phi(c-s)Q(s,z_s)dx_s \Big\| &\leq& K C_A (C_g \vee C_g^\prime) \Big[1+ |A|(b-a) \Big]e^{-\lambda_A (c-b)} \ltn x \rtn_{p{\rm - var},[a,b]} \times \notag\\
&& \times (1 + \ltn y^1 \rtn_{p{\rm - var},[a,b]}) \|z\|_{p{\rm - var},[a,b]}.
\end{eqnarray}
We need the following auxiliary result. 
\begin{lemma}\label{2sol}
	Assume that all the conditions in Theorem \ref{attractor} are satisfied. Let $y^1,y^2$ be two solutions of \eqref{fSDE0} and assign $z_t=y^2_t-y^1_t$ for all $t\geq 0$. \\
$(i)$ If $D_g$  is of Lipschitz continuity with Lipschitz constant $C'_g$, then
	\begin{equation}\label{zt2a}
	e^{\lambda n} \|z_{n}\|\leq C_A\|z_0\| + e^{\lambda_A }KC_A(1+|A|)(C_g\vee C'_g) \sum_{k=0}^{n-1} \ltn  x\rtn_{p{\rm -var},\Delta_k}e^{\lambda k}   \Big(1+\ltn y^1\rtn_{{p\rm-var},\Delta_k}\Big) \|  z\|_{p{\rm -var},\Delta_k}.
	\end{equation}
$(ii) $ If  $g(y) = C y + g(0)$ then	
	\begin{eqnarray}\label{zt2*}
	e^{\lambda n} \|z_{n}\|\leq C_A\|z_0\| + e^{\lambda_A }KC_A(1+|A|)C_g \sum_{k=0}^{n-1} \ltn  x\rtn_{p{\rm -var},\Delta_k}e^{\lambda k}    \|  z\|_{p{\rm -var},\Delta_k}.
	\end{eqnarray}
\end{lemma}
\begin{proof} 
	The arguments follow the proof of Lemma \ref{ytest} step by step, and apply Lemma \ref{Q}, Proposition \ref{YDEg} to obtain the estimate
	\[
		e^{\lambda_A t} \|z_t\|\leq C_A\|z_0\| + e^{\lambda_A t}\beta_t + L_f \int_0^t \Big(\|z_s\|+e^{\lambda_A s}\beta_s\Big) e^{L_f(t-s)}ds,
	\]
where $\beta_t= \|\int_0^t \Phi(t-s)Q(s,z_s)dx_s\|$ is estimated using \eqref{estQ}. The rest will be omitted. 
\end{proof}

\begin{theorem}\label{linear}
	Assume that $g(y) = Cy+g(0)$ is a linear map so that $C_g=|C|$. Then under the condition \eqref{criterion}, 
\[
	\lambda_A - C_A C_f> C_A(1+|A|) e^{\lambda_A+2(|A|+C_f)} \Big\{\Big[2(K+1)C_g\Gamma(p)\Big]^p + \Big[2(K+1)C_g\Gamma(p)\Big]\Big\},
\]
	the pullback attractor is a singleton, i.e. $\mathcal{A}(x) = \{a(x)\}$ almost surely. Moreover, it is also a forward singleton attractor. 
\end{theorem}
\begin{proof}
	The existence of the pullback attractor $\mathcal{A}$ is followed by Theorem \ref{attractor}. Take any two points $a_1,a_2\in \mathcal{A}(x) $. For a given $n \in \N$, assign $x^*:=\theta_{-n}x$  and consider the equation
	\begin{equation}\label{equ.star1}
	dy_t = [Ay_t + f(y_t)]dt + g(y_t)dx^*_t.
	\end{equation}
	Due to the invariance of $\mathcal{A}$ under the flow, there exist $b_1,b_2\in \mathcal{A}(x^*)$ such that $a_i=y_n(x^*,b_i)$.
	Put $z_t= z_t(x^*):= y_t(x^*,b_1)- y_t(x^*,b_2)$ then $z_n(x^*) =a_1- a_2$. By applying Lemma \ref{2sol} with $x$ replaced by $x^*$, and using Lemma \ref{Q} (ii), we can rewrite the estimates in \eqref{zt2*} as
	\begin{eqnarray}\label{zt2b1}
	e^{\lambda n} \|z_n\|&\leq &C_A\|z_0\| +M_1C_g \sum_{k=0}^{n-1} \ltn  x^*\rtn_{p{\rm -var},\Delta_k}e^{\lambda k}    \|  z\|_{p{\rm -var},\Delta_k}.
	\end{eqnarray}
	Meanwhile, using Corollary \ref{2sol1}$(ii)$ and Remark \ref{ysummary}, with $M_0$ in \eqref{M0} now equal to zero, we obtain
\begin{equation*}
\| z \|_{p{\rm -var},[a,b]} \leq \|z_a\|\Lambda_1(x^*,[a,b]), 
\end{equation*}
in which $\Lambda_1$ defined in \eqref{Lambda}.
As a result, \eqref{zt2b1} has the form
\begin{eqnarray*}
e^{\lambda n} \|z_n\|&\leq & 
C_A\|z_0\| +M_1C_g \sum_{k=0}^{n-1}  \ltn x^*\rtn_{p{\rm-var},[0,1]}\Lambda_1(x^*,[0,1]) e^{\lambda k}  \|z_k\|\\
&\leq&C_A\|z_0\| + M_1C_g\sum_{k=0}^{n-1}G(x^*,\Delta_k)e^{\lambda k} \|  z_k\|,
\end{eqnarray*}
where $G$ is defined in \eqref{G}.
Now applying the discrete Gronwall lemma \ref{gronwall}, we conclude that
\begin{eqnarray}\label{zn}
e^{\lambda n} \|z_n\|&\leq &C_A\|z_0\| \prod_{k=0}^{n-1}[1+M_1C_gG(x^*,\Delta_k)].
\end{eqnarray}
Similar to \eqref{Gnew}
\begin{eqnarray*}
&&\log (1+M_1C_gG(x^*,\Delta_k))\\
&\leq & C_Ae^{\lambda_A+2L}(1+|A|) \Big[\Big(2(K+1)C_g \ltn \theta_{k-n}x\rtn_{p{\rm-var},[0,1]} +\Big(2(K+1)C_g \ltn \theta_{k-n}x\rtn_{p{\rm-var},[0,1]}\Big)^p \Big].
\end{eqnarray*}
Therefore it follows from \eqref{zn} that
\allowdisplaybreaks
\begin{eqnarray*}\label{zlim}
\limsup \limits_{n \to \infty} \frac{1}{n}\log \|z_n\| &\leq &-\lambda + \limsup \limits_{n \to \infty} \frac{1}{n} \sum_{k=0}^{n-1} \log [1+M_1C_gG(x^*,\Delta_k)] \notag\\
&\leq & -\lambda +C_Ae^{\lambda_A+2L}(1+|A|)\Big[ 2(K+1)C_g\Gamma(p) + \Big(2(K+1)C_g\Big)^p\Gamma(p)^p\Big] \notag \\
&<&0 
\end{eqnarray*}
under the condition \eqref{criterion}. This follows that $\lim_{n\to\infty} \|a_1-a_2\|=0$ and $\mathcal{A}$ is a one point set almost surely. Similar arguments in the forward direction ($x^*$ is replaced by $x$) also prove that $\mathcal{A}$ is a forward singleton attractor almost surely.
\end{proof}

\begin{remark}\label{linearrem}
	As pointed out in the Introduction section, if we use the conjugacy transformation (developed in \cite{Sus78}, \cite{ImkSchm01}, \cite{KelSchm98}) of the form $y_t = e^{C \eta_t} z_t$, where the semigroup $e^{C t}$ is generated by the equation
	$\dot{u} = C u$ and $\eta$ is the unique stationary solution of the Langevin equation $d\eta = -\eta dt + dZ_t$ (with $Z$ is a scalar stochastic process), then the transformed system has the form
\[
	\dot{z}_t = e^{-C \eta_t} \Big[ A e^{C \eta_t} z_t + f(e^{C \eta_t} z_t) + \eta_t C e^{C \eta_t} z_t  \Big].
\]
	However, even in the simplest case $f \equiv 0$, there is no effective method to study the asymptotic stability of the non-autonomous linear stochastic system
	\begin{equation}\label{linconjugacy}
	\dot{z}_t = \Big( e^{-C \eta_t} A e^{C \eta_t} + \eta_t C \Big)z_t.
	\end{equation}
	An exception is when $A$ and $C$ are commute, since we could reduce system \eqref{linconjugacy} in the form
	\[
	\dot{z}_t = \Big(A+ \eta_t C \Big)z_t,
	\]
	thereby solve it explicitly as
	\[
	z_t = z_0 \exp \{A t + C \int_0^t \eta_s ds \} = z_0 \exp \{At - C(\eta_t - \eta_0 - Z_t + Z_0)  \}.
	\]
	In this case, the exponential stability is proved using the fact that $\exp \{- C(\eta_t - \eta_0 - Z_t + Z_0)  \}$ is tempered. However, since $A$ and $C$ are in general not commute, we can not apply the conjugacy transformation but should instead use our method described in Theorems \ref{attractor} and \ref{linear}.	
\end{remark}	

Next, motivated by \cite{hairer07}, we consider the case in which $g \in C^2_b$ and $C_g$ is also the Lipschitz constant of $Dg$. Notice that our conditions for $A$ and $f$ can be compared similarly to the dissipativity condition in \cite{hairer07}. However, unlike the probabilistic conclusion of existence and uniqueness of a stationary measure in \cite{hairer07}, we go a further step by proving that for $C_g$ small enough, the random attractor is indeed a singleton, thus the convergence to the attractor is in the pathwise sense and of exponential rate.

\begin{theorem}\label{gbounded}
	Assume that $g \in C^2_b$ with $\|g\|_\infty < \infty$, and denote by $C_g$ the Lipschitz constant of $g$ and $Dg$. Assume further that $\lambda_A>C_AC_f$ and
	\begin{equation}\label{newGamma}
	E \ltn Z \rtn^{\frac{4p(p+1)}{p-1}}_{p{\rm -var},[-r,r]}< \infty, \quad \forall r >0. 
	\end{equation}
	Then system \eqref{fSDE0} possesses a pullback attractor. Moreover, there exists a $\delta>0$ small enough such that for any $C_g\leq \delta$ the attractor is a singleton almost surely, thus the pathwise convergence is in both the pullback and forward directions.
\end{theorem}
\begin{proof}
	{\bf Step 1.} Similar to \cite[Proposition 4.6]{hairer07}, we prove that there exist a time $r>0$, a constant $\eta  \in (0,1)$, and an integrable random variable $\xi_r(x)$ such that
	\begin{equation}\label{ydissipative}
	\|y_r(x,y_0)\|^{p} \leq \eta \|y_0\|^{p} + \xi_r(x).
	\end{equation}
	First we fix $r>0$ and consider $\mu,h$ as defined in Lemma \ref{h} on $[0,r]$, i.e $\mu_t$ is the solution of the deterministic system $\dot{\mu} = A\mu + f(\mu)$ which starts at $\mu_0 = y_0$ and $h_t=y_t-\mu_t$.  Then using \eqref{yaux} 
\[
	\|\mu_r\| \leq C_A \|y_0\| e^{-\lambda r} + C_A \frac{\|f(0)\|}{\lambda}. 
\]
	On the other hand, due to \eqref{hinf} in Corollary \ref{h} and \eqref{Nest},
\[
\|h_r\|\leq\|h\|_{\infty,[0,r]} \leq \xi_0 (x) (1+  \|y_0\|^\beta),
\]
where  $\beta = \frac{1}{p}$ and $\xi_0$ is a polynomial of $\ltn x\rtn_{p{\rm-var},[0,r]}$ of the form
	\[
	\xi_0(x) = D \ltn x\rtn_{p{\rm-var},[0,r]}\Big(1+\ltn x\rtn^p_{p{\rm-var},[0,r]} \Big),
	\]
	where $D$ is a constant.

	 Now for $\epsilon>0$ small enough, we apply the convex inequality and Young inequality to conclude that 
	 \allowdisplaybreaks
	\begin{eqnarray}\label{hest3}
	\|y_r\|^{2p} \leq (\|h_r\| + \|\mu_r\|)^{2p} \leq (1+\epsilon)^{2(2p-1)} \Big[(C_Ae^{-\lambda r})^{2p}  + \epsilon \beta \Big] \|y_0\|^{2p} + \xi_r(x),
	\end{eqnarray}
	where 
	\[
	\xi_r(x) = \xi_r \Big(\frac{1}{\epsilon},x\Big) \leq  \frac{D}{\epsilon^{2p-1+\frac{2p}{p-1}}}\Big( 1 +  \ltn x\rtn_{p{\rm-var},[0,r]}^{\frac{2p^2(p+1)}{p-1}} \Big) 
	\]
	for some generic constant $D$ (depends on $r$). Thus $\xi_r$ is integrable due to $\frac{2p^2(p+1)}{p-1} \leq \frac{4p(p+1)}{p-1}$ and \eqref{newGamma}. By choosing $r >0$ large enough and $\epsilon \in (0,1)$ small enough such that  
	\[
	C_Ae^{-\lambda r} < 1 \quad \text{and} \quad \eta : = (1+\epsilon)^{2(2p-1)} \Big[(C_Ae^{-\lambda r})^{2p}  + \epsilon \beta \Big]<1
	\]
	we obtain \eqref{ydissipative}.
	
	{\bf Step 2.} Next, for simplicity we only estimate $y$ at the discrete times $nr$ for $n \in \N$, the estimate for $t \in [nr,(n+1)r]$ is similar to \eqref{y0}. From \eqref{ydissipative}, it is easy to prove by induction that
\[
\|y_{nr}(x,y_0)\|^{2p} \leq \eta^n \|y_0\|^{2p} + \sum_{i = 0}^{n-1} \eta^i \xi_r(\theta_{(n-i)r}x), \quad \forall n \geq 1;
\]
thus for $n$ large enough 
\[
\|y_{nr}(\theta_{-nr}x,y_0)\|^{2p} \leq \eta^n \|y_0\|^{2p} + \sum_{i = 0}^n \eta^i \xi_r(\theta_{-ir}x) \leq 1 + \sum_{i =0}^\infty \eta^i \xi_r(\theta_{-ir}x) =: R_r(x).
\]
In this case we could choose $\hat{b}(x)$ in \eqref{bhat} to be $\hat{b}(x)=R_r(x)^{\frac{1}{2p}}$ so that there exists a pullback absorbing set $\cB(x) = B(0,\hat{b}(x))$ containing our random attractor $\mathcal{A}(x)$. Moreover, due to the integrability of $\xi_r(x)$, $R_r(x)$ is also integrable with $\E R_r = 1 + \frac{\E \xi_r}{1-\eta}$.

	{\bf Step 3.} Now back to the arguments in the proof of Theorem \ref{linear} and note that $Dg$ is also globally Lipschitz with the same constant $C_g$. Using Lemma \ref{Q} $(i)$ and rewriting \eqref{zt2a} in Lemma \ref{2sol} for $x^*$ yields
	\begin{equation}\label{zt2b}
	e^{\lambda n} \|z_{n}\|\leq C_A\|z_0\| + M_1C_g \sum_{k=0}^{n-1} \ltn  x^*\rtn_{p{\rm -var},\Delta_k}e^{\lambda k}   \Big(1+\ltn y^1\rtn_{{p\rm-var},\Delta_k}\Big) \|  z\|_{p{\rm -var},\Delta_k}
	\end{equation}
where the $p$-variation norm of $z$ can be estimated, due to Corollary \ref{2sol1}$(i)$, as 
	\begin{eqnarray}
	\| z \|_{p{\rm -var},\Delta_k} &\leq& \Big(N^\prime_{\Delta_k}(x^*)\Big)^{\frac{p-1}{p}} 2^{N^\prime_{\Delta_k}(x^*)}e^{2L }\|z_a\|,\\
\text{with}\qquad \qquad N^\prime_{\Delta_k}(x^*) &\leq& 1+[2(K+1)C_g]^p \ltn x^* \rtn_{p{\rm -var},\Delta_k}^p(1+ \ltn y^1 \rtn_{p{\rm -var},\Delta_k})^p.\label{Nprime2}
	\end{eqnarray}
	This together with \eqref{zt2b} derives 
\begin{eqnarray}
e^{\lambda n} \|z_{n}\|&\leq& C_A \|z_0\| +  \sum_{k=0}^{n-1} I_k e^{\lambda k} \|z_{k}\| \label{zest2} \\
\text{where}\quad I_k &=& M_1C_g \ltn  x^*\rtn_{p{\rm -var},\Delta_k} \Big(1+\ltn y^1\rtn_{{p\rm-var},\Delta_k}\Big)\Big(N^\prime_{\Delta_k}(x^*)\Big)^{\frac{p-1}{p}} 2^{N^\prime_{\Delta_k}(x^*)}e^{2L }.\notag
\end{eqnarray}
By applying the discrete Gronwall lemma \ref{gronwall}, we obtain
\begin{eqnarray}\label{logz}
\|z_{n}\|&\leq &C_A\|z_0\|  e^{-\lambda n}\prod_{k=0}^{n-1} (1+I_k) \notag\\
\text{which yields}\qquad \frac{1}{n}\log \|z_{n}\|&\leq & \frac{1}{n}\log \Big(2C_A \hat{b}(x^*)\Big)  -\lambda  + \frac{1}{n}\sum_{k=0}^{n-1} \log (1+I_k).
\end{eqnarray}
On the other hand, due to \eqref{Nest}, \eqref{esty2} and \eqref{bhat}, it is easy to prove with a generic constant $D$ that
\[
		\ltn y^1 \rtn_{p{\rm -var},\Delta_k} \leq D\Big(1+\ltn \theta_{(k-n)}x \rtn^{2p-1}_{p{\rm -var},[0,1]}\Big)  \Big( 1+ \hat{b}(\theta_{(k-n)}x) + \ltn  \theta_{(k-n)}x \rtn_{p{\rm -var},[0,1]} \Big),
\]
thus	
\begin{eqnarray*}
&&	[2(K+1)C_g] \ltn x^* \rtn_{p{\rm -var},\Delta_k}(1+ \ltn y^1 \rtn_{p{\rm -var},\Delta_k}) \\
&\leq&  D \Big\{\ltn \theta_{(k-n)}x \rtn_{p{\rm -var},[0,1]}+ \ltn \theta_{(k-n)}x \rtn^2_{p{\rm -var},[0,1]} +\ltn \theta_{(k-n)}x \rtn^{2p}_{p{\rm -var},[0,]}+ \ltn \theta_{(k-n)}x \rtn^{2p+1}_{p{\rm -var},[0,]}\Big \} [1+ \hat{b}(\theta_{(k-n)}x) ] \\
	&=:& \hat{F} (\theta_{(k-n)}x).
	\end{eqnarray*}
All together, $I_k$ is bounded from above by
\begin{eqnarray}\label{Ik}
	I_k \leq D C_g\hat{F}(\theta_{(k-n)}x)\left[1+ \hat{F}^{p-1} (\theta_{(k-n)}x)\right] \exp\left\{\log 2\left[1+ \hat{F}^{p} (\theta_{(k-n)}x)\right]\right\}
	\end{eqnarray}
where the right hand side of \eqref{Ik} is a function of $\theta_{(k-n)}x$.
The ergodic Birkhorff theorem is then applied for \eqref{logz}, so that
	\begin{equation}\label{exponent1}
	\limsup \limits_{n \to \infty} \frac{1}{n}\log \|z_{n}\| 
	\leq  -\lambda + E \log \left\{1+ D C_g \Big[\hat{F} (x) + \hat{F}^p (x)\Big]e^{D\hat{F}^p (x) \log 2} \right\}\quad \text{a.s.} 
	\end{equation}
Apply the inequalities
	\begin{eqnarray*}
		\log(1 + x + y) &\leq& \log (1+x) + \log (1+y),\\
		\log(1+xy) &\leq& \log (1+x) + \log (1+y),\\
		\log (1+ x e^y) &\leq& x + y,\qquad x,y\geq 0,
	\end{eqnarray*}
it follows that
	\begin{eqnarray}\label{exponent2}
\log \Big\{1+ D C_g \Big[\hat{F} (x) + \hat{F}^p (x)\Big]e^{D\hat{F}^p (x) \log 2} \Big\}\leq D \Big(1+ \hat{F}(x)+\hat{F}^p(x) \Big).
	\end{eqnarray}
To estimate $\hat{F}^p(x)$, we apply Cauchy and Young inequalities to obtain, up to a generic constant $D>0$,
	\begin{eqnarray*}
		\hat{F}^p(x)
		&\leq&D\Big[1+ \ltn x \rtn^{2}_{p{\rm -var},[0,1]} +\ltn x \rtn^{4}_{p{\rm -var},[0,1]}+\ltn x \rtn^{4p}_{p{\rm -var},[0,1]}+ \ltn x \rtn^{4p+2}_{p{\rm -var},[0,1]} +\hat{b}^{2}(x)\Big]^p\notag\\
		&\leq&D\Big[1+ \ltn x \rtn^{2p}_{p{\rm -var},[0,1]} +\ltn x \rtn^{4p}_{p{\rm -var},[0,1]}+\ltn x \rtn^{4p^2}_{p{\rm -var},[0,1]}+ \ltn x \rtn^{4p^2+2p}_{p{\rm -var},[0,1]} +\hat{b}^{2p}(x)\Big].
	\end{eqnarray*}
	Hence the right hand side in the last line of \eqref{exponent2} is integrable due to \eqref{bhat} and the integrability of $\ltn x \rtn^{\frac{4p(p+1)}{p-1}}_{p{\rm -var},[-1,1]}$ in \eqref{newGamma} and of $\hat{b}(x)^{2p}$ in Step 2. On the other hand, the expression under the expectation of \eqref{exponent1} tends to zero a.s. as $C_g$ tends to zero. Due to the Lebesgue's dominated convergence theorem, the expectation converges to zero as $C_g$ tends to zero. As a result, there exists $\delta$ small enough such that for $C_g < \delta$ we have $\|z_{n}\| = \|a_1 - a_2\| \to 0$ as $n$ tends to infinity exponentially with the uniform convergence rate in \eqref{exponent1}. This proves $a_1 \equiv a_2$ a.s. and $\mathcal{A}$ is a singleton. 
	
{\bf Step 4.} Let $y^1_t= y(t,x,a(x))$, $y^2_t=y(t,x,y_0(x))$ be the solutions starting from $a(x)$, $y_0(x)$ respectively at $t=0$.  Since $\mathcal{A}$ is invariant, $y^1_t= a(\theta_tx)$. By repeating the arguments in Step 3 ($x^*$ is replaced by $x$), we conclude that $ \mathcal{A} (x)=\{a(x)\}$ is also a forward attractor. 
\end{proof}
\begin{corollary}\label{continuityattractor}
Denote by $\mu^*$ the unique equilibrium of the deterministic system \eqref{mu.equ}. Assume that $\|g\|_\infty \leq C_g $ and the assumptions of Theorem \ref{gbounded} hold so that there exists a singleton attractor $\mathcal{A}(x) = \{a(x)\}$ for $C_g < \delta$ small enough. Then 
\begin{equation}\label{attractorcont}
\lim \limits_{C_g \to 0} \|a(x) - \mu^*\| =0 \quad \text{a.s., and}\quad \lim \limits_{C_g \to 0} \E \|a(x) - \mu^*\|^{2p} =0.   
\end{equation}
\end{corollary}
\begin{proof}
Take a solution $y_t(x,\mu^*)$ and consider the difference $h_t=h_t(x,\mu^*) := y_t(x,\mu^*) - \mu^*$ for $t\geq 0$, then $h_0 =0$ and $h$ satisfies the equation
\[
dh_t = [f(h_s+\mu^*)-f(\mu^*)] dt + g(h_t+\mu^*) d x_t = \bar{f}(h_t) dt + \bar{g}(h_t) d x_t.
\] 
Similar to \eqref{esty2}, it is easy to prove that
\begin{equation}\label{esth}
	\| h \|_{p{\rm -var},[a,b]} \leq \Big[\|h_a\| + 2 C_g (1+  \ltn x\rtn_{p{\rm-var},[a,b]})N_{[a,b]}(x)\Big] e^{2 L (b-a)}N^{\frac{p-1}{p}}_{[a,b]}(x).
	\end{equation}
On the other hand, $h$ also satisfies a similar equation to \eqref{variation}, which is  
\[
h_t = \int_0^t\Phi (t-s)\bar{f}(h_s) ds + \int_0^t\Phi (t-s)\bar{g}(h_s) d x_s,\quad \forall t\geq 0.
\] 
where a similar computation to \eqref{int2} shows that
\[
		\Big\|\int_a^{b}\Phi (c-s)\bar{g}(h_s) d  x_s\Big\| \leq KC_A \Big[1+|A|(b-a)\Big]\ltn  x\rtn_{p{\rm -var},[a,b]}e^{-\lambda_A(c-b)} \Big[C_g \| h\|_{p{\rm -var},[a,b]}+C_g\Big].
\]
As a result, one can prove a similar estimate to \eqref{yt} that
\begin{eqnarray}\label{ht}
	\|h_n\|e^{\lambda n}
	\leq C_A \|h_0\| + \sum_{k=0}^{n-1} e^{\lambda_A}KC_A(1+|A|)\ltn  x\rtn_{p{\rm -var},\Delta_k}e^{\lambda k} \Big[C_g \| h\|_{p{\rm -var},\Delta_k}+C_g\Big].
	\end{eqnarray}
Using \eqref{ht}, \eqref{esth} and following the proof of Theorem \ref{attractor} step by step, we obtain	
\begin{eqnarray}\label{h}
	\|h_n(x,y_0)\|&\leq &C_A\|h_0\|e^{-\lambda n}  \prod_{k=0}^{n-1} \Big[1+M_1 C_g G(x,\Delta_k)\Big]\notag\\ &&+DC_g  \sum_{k=0}^{n-1} e^{-\lambda (n-k)} H(x,\Delta_k) \prod_{j=k+1}^{n-1} \Big[1+M_1C_g G(x,\Delta_j)\Big].
\end{eqnarray}	
for some generic constant $D$. Replacing $x$ by $\theta_{-n}x$ and letting $n$ tend to infinity, we follow from \eqref{h} that 
\[
\lim \limits_{n \to \infty} \|h_n(\theta_{-n}x,\mu^*)\| = \lim \limits_{n \to \infty} \|y_n(\theta_{-n}x,\mu^*) - \mu^*\| \leq D C_g b(x)
\]
for some random variable $b(\cdot)$ that is finite almost sure. On the other hand, Theorem \ref{gbounded} shows that $y_n(\theta_{-n}x,\mu^*) \to a(x)$ a.s. as $n \to \infty$. Hence $\|a(x) - \mu^*\|  \leq DC_g b(x)$ a.s., which shows that $a(x)$ converges to $\mu^*$ a.s. as $C_g \to 0$. 

 Next, as shown in Step 1 of the proof of Theorem \ref{gbounded}, $\E \|a(\cdot)\|^{2p} < \infty$, thus $\E \|a(\cdot) - \mu^*\|^{2p} < 2^{2p-1} \Big(\E \|a(\cdot)\|^{2p} + (\mu^*)^{2p} \Big) <\infty$.   Due to the Lebesgue's dominated convergence theorem, $a(x)$ also converges to $\mu^*$ in $2p$ - moment as $C_g$ tends to zero, which proves \eqref{attractorcont}.
\end{proof}

\begin{example}	[Inverted pendulum with stochastic excitation]\label{ex1}
	Following \cite{shaikhet}, \cite{floris} and the references therein, we study the dynamics of an inverted pendulum with a point mass at the top and a rotational spring at the base. Assume that the bar is massless and the base is subjected to a vertical acceleration $A_t$, which is supposed to be a stochastic process $\sigma_1 \dot{\xi}^1_t$ for a parameter $\sigma_1 \geq 0$. Then the dynamics of the angle of rotation $\vartheta$ follows a stochastic differential equations 
	\begin{equation}\label{pendulum1}
	m\ddot{\vartheta} + 2b \dot{\vartheta}  - \frac{m}{L}[g + \sigma_1 \dot{\xi}^1_t]\sin \vartheta + k \vartheta = 0,
	\end{equation}
where $m$ is the tip mass, $b$ is the viscous damping coefficient, $L$ is the bar's length, $g$ is the gravitational constant and $k$ is the excitation intensity. In practice, the coefficients $b,k$ can also be of stochastic form $b + \sigma_2 \dot{\xi}^2_t$, $k + \sigma_3 \dot{\xi}^3_t$, and there might be another small external stochastic excitation $\sigma_4 \dot{\xi}^4_t$ in the right hand side of equation \eqref{pendulum1}, for $\sigma_2, \sigma_3, \sigma_4 \geq 0$. Hence the generalized form of \eqref{pendulum1} is
	\begin{equation}\label{pendulum2}
	\ddot{\vartheta} + \frac{2(b+\sigma_2 \dot{\xi}^2_t)}{m} \dot{\vartheta}  - \frac{1}{L}[g + \sigma_1 \dot{\xi}^1_t]\sin \vartheta + \frac{k+ \sigma_3 \dot{\xi}^3_t}{m} \vartheta = \frac{\sigma_4}{m} \dot{\xi}^4_t.
	\end{equation}
 To interpret equations \eqref{pendulum1} and \eqref{pendulum2}, we assume further that $\dot{\xi}^i_t = \frac{d z^i_t}{dt}, i = 1,\ldots,4$, where $Z=(z^1,z^2,z^3,z^4) \in \R^4$ is a stationary stochastic process statisfying  (${\textbf H}_3$) and condition \eqref{newGamma}. For instance, we can assume that $z^i$ are mutually independent fractional Brownian motions $B^{H_i}$ where $H_i > \frac{1}{2}$ for $i = 1\ldots 4$. Then equation \eqref{pendulum2} is understood as a two-dimensional controlled differential equation
 \begin{eqnarray}\label{pendulum3}
   d \vartheta_t &=& \bar{\vartheta}_t dt \notag \\ 
   d \bar{\vartheta}_t &=&\Big[ - \frac{k}{m} \vartheta_t -  \frac{2b}{m} \bar{\vartheta}_t + \frac{g}{L} \sin \vartheta_t \Big]dt + \frac{\sigma_1 }{L} \sin \vartheta_t dz^1_t - \frac{2\sigma_2}{m}\bar{\vartheta}_t dz^2_t - \frac{\sigma_3}{m} \vartheta dz^3_t+  \frac{\sigma_4}{m} dz^4_t,
 \end{eqnarray}
 which has the form \eqref{stochYDE} with $y = (\vartheta, \bar{\vartheta})^{\rm T}$ and
\begin{equation}\label{pendulum4}     	
A = \left(\begin{matrix}
	0 & 1\\ - \frac{k}{m}& -\frac{2b}{m} \end{matrix} \right),\quad  f(y) = \left(\begin{matrix}
	0\\   \frac{g}{L} \sin \vartheta_t
	\end{matrix} \right), \quad g(y) =  \left(\begin{matrix}
	 0 & 0 & 0&0\\ \frac{\sigma_1 }{L} \sin \vartheta_t & -\frac{2\sigma_2}{m}\bar{\vartheta}  & - \frac{\sigma_3}{m} \vartheta & \frac{\sigma_4}{m} \end{matrix} \right).
	\end{equation}
It is easy to check that matrix $A$ has two eigenvalues of negative real parts,  $f$ is globally Lipschitz continuous with $C_f = \frac{g}{L}$, and $g \in C^2 (\R^2,\R^{2\times 4})$ with $C_g = \frac{\sigma_1}{L} \vee \frac{2\sigma_2}{m} \vee \frac{\sigma_3}{m}$. Hence, provided that $\lambda = \lambda_A - C_A C_f >0$, we consider three cases.
\begin{itemize}
	\item If $\sigma_4 = 0$, then due to Corollary \ref{onepoint} the singleton attractor is the trivial solution if we choose $\sigma_1, \sigma_2, \sigma_3$ small enough such that condition \eqref{criterion} is satisfied. 
	\item If $\sigma_2 = \sigma_3 = 0$, then only two noises $z^1,z^4$ are in effect and $g$ has a reduced form which is in $C^2_b(\R^2,\R^{2\times 2})$. We then apply Theorem \ref{gbounded} to conclude that, by choosing $\sigma_1$ small enough, all solutions of \eqref{pendulum3} converge a.s. in the pullback sense to the singleton attractor of the random dynamical system generated by \eqref{pendulum3}. Using Corollary \ref{continuityattractor}, we can choose $\sigma_1$ and $\sigma_4$ small enough so that this singleton attractor is very close to the trivial solution of the unperturbed system. 
	\item If $\sigma_1 = 0$, then the three noises $\sigma_2, \sigma_3, \sigma_4$ are in effect and $g$ has a reduced form of linear type $C y + g(0)$ where $C \in \R^{2 \times 2 \times 3}$. We then apply Theorem \ref{linear} to conclude that, by choosing $\sigma_2, \sigma_3$ small enough such that condition \eqref{criterion} is satisfied, all solutions of \eqref{pendulum3} converge a.s. in the pullback sense to the singleton attractor of the random dynamical system generated by \eqref{pendulum3}. 	
\end{itemize}
\end{example}
\section*{Appendix}
\begin{lemma}[Discrete Gronwall Lemma]\label{gronwall}
	Let $a$ be a non negative constant and $u_n, \alpha_n,\beta_n$ be  nonnegative sequences satisfying 
	\[
	u_n\leq a + \sum_{k=0}^{n-1} \alpha_ku_k + \sum_{k=0}^{n-1} \beta_k,\;\; \forall n \geq 1
	\]
	\begin{equation}\label{estu}
\text{then}\qquad	u_n\leq \max\{a,u_0\}\prod_{k=0}^{n-1} (1+\alpha_k) + \sum_{k=0}^{n-1}\beta_k\prod_{j=k+1}^{n-1}(1+\alpha_j), \quad \forall n \geq 1.
	\end{equation}
\end{lemma}
\begin{proof}
	Put 
	\begin{eqnarray*}
		S_n:= a + \displaystyle\sum_{k=0}^{n-1} \alpha_ku_k + \sum_{k=0}^{n-1} \beta_k, \quad T_n: = \max\{a,u_0\} \displaystyle\prod_{k=0}^{n-1} (1+\alpha_k) + \sum_{k=0}^{n-1}\beta_k\prod_{j=k+1}^{n-1}(1+\alpha_j).
	\end{eqnarray*}
	We will prove by induction that $S_n\leq T_n$ for all $n\geq 1$. Namely, the statement holds for $n=1$ since $S_1= a+\alpha_0 u_0+\beta_0 \leq \max\{a,u_0\}(1+\alpha_0)  + \beta_0 =T_1$.
	
	We assume that $S_n\leq T_n$ for $n\geq 1$, then due to the fact that $u_n \leq S_n$ we obtain
	\allowdisplaybreaks
	\begin{eqnarray*}
		S_{n+1}	&=&  a + \displaystyle\sum_{k=0}^{n-1} \alpha_ku_k + \sum_{k=0}^{n-1} \beta_k + \alpha_n u_n + \beta_n	= S_n + \alpha_n u_n + \beta_n\notag\\
		&\leq & S_n + \alpha_n S_n + \beta_n \leq  T_n (1+ \alpha_n ) + \beta_n\notag\\
		&\leq &  \left[\max\{a,u_0\} \displaystyle\prod_{k=0}^{n-1} (1+\alpha_k) + \sum_{k=0}^{n-1}\beta_k\prod_{j=k+1}^{n-1}(1+\alpha_j) \right] (1+ \alpha_n ) +\beta_n\\
		&\leq &  \max\{a,u_0\} \displaystyle\prod_{k=0}^{n} (1+\alpha_k) + \sum_{k=0}^{n}\beta_k\prod_{j=k+1}^{n-1}(1+\alpha_j) = T_{n+1}.
	\end{eqnarray*}
	Since $u_n\leq S_n$, \eqref{estu} holds.
\end{proof}

\section*{Acknowledgments}
This work was supported by the Max Planck Institute for Mathematics in the Science (MIS-Leipzig). P.T. Hong would like to thank the IMU Breakout Graduate Fellowship Program for the financial support.

\end{document}